\documentclass{aims}
\usepackage{amsmath}
  \usepackage{paralist}
  \usepackage[pagewise]{lineno}
  \usepackage{physics}
  \usepackage{dsfont}
  \usepackage{csquotes}
  \usepackage{graphics} %% add this and next lines if pictures should be in esp format
  \usepackage{epsfig} %For pictures: screened artwork should be set up with an 85 or 100 line screen
\usepackage{graphicx}  \usepackage{epstopdf}%This is to transfer .eps figure to .pdf figure; please compile your paper using PDFLeTex or PDFTeXify.
 \usepackage[colorlinks=true]{hyperref}
\hypersetup{urlcolor=blue, citecolor=red}

\def\currentvolume{X}
 \def\currentissue{X}
  \def\currentyear{200X}
   \def\currentmonth{XX}
    \def\ppages{X--XX}

\newtheorem{theorem}{Theorem}[section]

\newtheorem{assumption}{Assumption} %[section]

\newtheorem{lemma}{Lemma}[section]
\newtheorem{proposition}{Proposition}[section]

\newtheorem{definition}{Definition}[section]
\newtheorem{remark}{Remark}[section]

\DeclareMathOperator*{\esssup}{ess\,sup}
\DeclareMathOperator*{\essinf}{ess\,inf}

\title[Path-Dependent Stochastic Differential Games] %Use the shortened version of the full title
      {State and Control Path-Dependent Stochastic Zero-Sum Differential Games: Viscosity Solutions of Path-Dependent Hamilton-Jacobi-Isaacs Equations}

\author[Jun Moon]{}

 %Differential games and control

 \email{jmoon12@uos.ac.kr}
\thanks{This research was supported in part by the National Research Foundation of Korea (NRF) Grant funded by the Ministry of Science and ICT, South Korea (NRF-2017R1E1A1A03070936, NRF-2017R1A5A1015311), and in part by Institute for Information \& communications Technology Promotion (IITP) grant funded by the Korea government (MSIT), South Korea (No. 2018-0-00958).}

\thanks{$^*$ Corresponding author: Jun Moon}

\begin{document}
\maketitle

% Enter the first author's name and address:
\centerline{\scshape Jun Moon$^*$}
\medskip
{\footnotesize
% please put the address of the first author
 \centerline{School of Electrical and Computer Engineering}
%   \centerline{Other lines}
 \centerline{University of Seoul, Seoul 02504, South Korea}
} % Do not forget to end the {\footnotesize by the sign }

%\medskip

%\centerline{\scshape First-name2 last-name2 and First-name3
%last-name3}
%\medskip
%{\footnotesize
% % please put the address of the second  and third author
% \centerline{ First line of the address of the second author}
%   \centerline{Other lines}
%   \centerline{Springfield, MO 65810, USA}
%}

\bigskip

% The name of the associate editor will be entered by an editorial staff
% "Communicated by the associate editor name" is not needed for special issue.
 \centerline{(Communicated by the associate editor name)}

%The abstract of your paper
\begin{abstract}
	In this paper, we consider state and control path-dependent stochastic zero-sum differential games, where the dynamics and the running cost include both state and control paths of the players. Using the notion of nonanticipative strategies, we define lower and upper value functionals, which are functions of the initial state and control paths of the players. We prove that the value functionals satisfy the dynamic programming principle. The associated lower and upper Hamilton-Jacobi-Isaacs (HJI) equations from the dynamic programming principle are state and control path-dependent nonlinear second-order partial differential equations. We apply the functional It\^o calculus to prove that the lower and upper value functionals are viscosity solutions of (lower and upper) state and control path-dependent HJI equations, where the notion of viscosity solutions is defined on a compact subset of an $\kappa$-H\"older space introduced in \cite{Tang_DCD_2015}. Moreover, we show that the Isaacs condition and the uniqueness of viscosity solutions imply the existence of the game value. For the state path-dependent case, we prove the uniqueness of classical solutions for the (state path-dependent) HJI equations.
\end{abstract}

\section{Introduction}\label{Section_1}

Since the seminal papers by Friedman \cite{Friedman_1972} and Fleming and Souganidis \cite{Fleming_1989}, the study of two-player stochastic zero-sum differential games (SZSDGs) and nonzero-sum stochastic differential games (SDGs) has grown rapidly in various aspects; see \cite{Basar2, Bayraktar_2013_SICON, Buckdahn_Cardaliaguet_DGAA_2011, Buckdahn_SICON_2004, Buckdahn_SICON_2008, Cardaliaguet_IGTR_2008, Djehiche_AMO_2018, Karoui_2003, Hamadene_SSR_1995, Li_AMO_2015, Moon_TAC_Risk_2019, Sun_Yong_SICON_2014, YU_SICON_2015} and the references therein. Specifically, Friedman in \cite{Friedman_1972} considered SDGs with classical (or smooth) solutions of the associated partial differential equation (PDE) from dynamic programming to prove the existence of the Nash equilibrium and the game value. Fleming and Souganidis in \cite{Fleming_1989} studied SZSDGs in the Markovian framework with \emph{nonanticipative strategies}. They proved that the lower and upper value functions are unique viscosity solutions (in the sense of \cite{Crandall_AMS_1992_Viscosity}) for lower and upper Hamilton-Jacobi-Isaacs (HJI) equations obtained from dynamic programming, which are nonlinear second-order partial differential equations (PDEs). They also showed the existence of the game value under the Isaacs condition. Later, the results of \cite{Fleming_1989} were extended by Buckdahn and Li in \cite{Buckdahn_SICON_2008}, who defined the objective functional by the backward stochastic differential equation (BSDE). They used the \emph{backward semi-group} associated with the BSDE introduced in \cite{Peng_SSR_1992} to obtain the generalized results of \cite{Fleming_1989}. The weak formulation of SZSDGs and SDGs with random coefficients was considered in \cite{Karoui_2003, Hamadene_SSR_1995}, where the existence of the open-loop type saddle-point equilibrium as well as the game value was established.

Recently, state path-dependent SZSDGs have been studied extensively in the literature to consider a general class of SZSDGs including SZSDGs with delay in the state variable. This extends the results in \cite{Buckdahn_SICON_2008, Fleming_1989} to the non-Markovian framework. Unlike \cite{Buckdahn_SICON_2008, Fleming_1989}, for the path-dependent or non-Markovian case, the associated (lower and upper) HJI equations obtained from dynamic programming are the so-called (state) path-dependent PDEs (PPDEs) defined on a space of continuous functions, which is an infinite dimensional Banach space. Hence, the approach for the Hilbert space in \cite{Crandall_Lions_JFA_1985, Lions_AM_1988} cannot be applied to show the existence (and uniqueness) of viscosity solutions. In \cite{Pham_SICON_2014, Possamai_arXiv_2018}, state path-dependent SZSDGs in weak formulation were studied, where the players are restricted to observe the state feedback information. The existence of viscosity solutions for state path-dependent HJI equations was shown in \cite{Pham_SICON_2014, Possamai_arXiv_2018} in the sense of  \cite{Ekren_AP_2014, Ekren_AP_2016_Part_I, Ekren_AP_2016}, which involves a nonlinear expectation in the corresponding semi-jets. For the uniqueness, \cite{Pham_SICON_2014} imposed the assumption on the maximum dimension of the state space ($n\leq 2$) and the nondegeneracy condition of the diffusion coefficient (see \cite[Section 6]{Pham_SICON_2014} and \cite[Remark 3.7]{Possamai_arXiv_2018}). Note that \cite{Possamai_arXiv_2018} did not prove the uniqueness of viscosity solutions. As mentioned in  \cite[p. 10]{Possamai_arXiv_2018}, the major motivation of SZSDGs in weak formulation is to study the existence of the saddle-point equilibrium; however, it requires more stringent assumptions on the coefficients than SZSDGs in strong formulation. Zhang in \cite{Zhang_SICON_2017} studied path-dependent SZSDGs in strong formulation, where the existence of the game value and the approximated saddle-point equilibrium, both under the Isaacs condition, were established via the approximating technique of the (lower and upper) state path-dependent HJI equations.\footnote{Note that \cite{Zhang_SICON_2017} did not consider the existence and uniqueness of (classical or viscosity) solutions of state path-dependent HJI equations.}

In this paper, we consider state and control path-dependent stochastic zero-sum differential games (SZSDGs), where the dynamics and the running cost are dependent on both the state and control path of the players. Note that the paper can be viewed as a generalization of \cite{Zhang_SICON_2017} to the state and control path-dependent case, and of \cite{Saporito_SICON_2019} to the two-player SZSDG framework. We mention that the viscosity solution of state path-dependent HJI equations was not studied in \cite{Zhang_SICON_2017}. In \cite{Saporito_SICON_2019}, state and control path-dependent stochastic control problem and (nonzero-sum) differential games were considered, where the existence of classical (smooth) solutions of the corresponding state and control path-dependent Hamilton-Jacobi-Bellman (HJB) equation was assumed to establish the verification theorem. Note also that SZSDGs with weak formulation in \cite{Pham_SICON_2014, Possamai_arXiv_2018} did not consider the control path-dependent case.

By using the notion of nonanticipative strategies, the lower and upper value functionals are defined, whereby these are functions of the initial state and control paths of the players. Then by using the semigroup operator associated with the BSDE objective functional, we prove that the (lower and upper) value functionals satisfy the dynamic programming principle, which is the recursive-type value iteration algorithm. In the proof of the dynamic programming principle, the separability of the space of c\`adl\`ag functions is essential, which is guaranteed with the Skorohod metric. The dynamic programming principle also leads to the continuity property of the (lower and upper) value functionals in their arguments.

The associated lower and upper state and control path-dependent Hamilton-Jacobi-Isaacs (PHJI) equations from the dynamic programming principle are state and control path-dependent nonlinear second-order PDEs (PPDEs), whose structures are fundamentally different from those of PPDEs or state path-dependent HJI equations in \cite{Ekren_AP_2014, Ekren_AP_2016_Part_I, Ekren_AP_2016, Peng_Gexpectation_arXiv_2011, Pham_SICON_2014, Possamai_arXiv_2018, Tang_DCD_2015}. In particular, the time derivative term also depends on the control of the players, which is included in $\sup_{v \in \mathrm{V}}	 \inf_{u \in \mathrm{U}} $ of the lower PHJI equation and $\inf_{u \in \mathrm{U}} \sup_{v \in \mathrm{V}}$ of the upper PHJI equation (see (\ref{eq_5_2}) and (\ref{eq_5_3})). We apply the functional It\^o calculus introduced in \cite{Cont_JFA_2010, Cont_AOP_2013, Dupire_2009} to prove that the lower and upper value functionals are viscosity solutions of the (lower and upper) PHJI equations, where the notion of viscosity solutions is defined on a compact subset of an $\kappa$-H\"older space introduced in \cite{Tang_DCD_2015}. Specifically, the notion of viscosity solutions is defined on a compact set $\mathbb{C}^{\kappa,\mu,\mu_0}$, where $\kappa \in (0,\frac{1}{2})$ and $\mu,\mu_0 > 0$, which provides the precise estimate between the initial state path and its perturbed one. This initial state path perturbation is essential to prevent starting the dynamic programming at the boundary of $\mathbb{C}^{\kappa,\mu,\mu_0}$ (see \cite[Remark 6]{Tang_DCD_2015}). Then using the functional It\^o calculus and the dynamic programming principle, we show that the (lower and upper) value functionals are viscosity solutions of the corresponding PHJI equations. In our definition of viscosity solutions, the  \emph{predictable dependence} condition for test functions is essential to handle the control path-dependent nature of the problem, where the similar condition was also introduced in \cite{Cont_JFA_2010, Saporito_SICON_2019}.

We also show that if the state and control path-dependent Isaacs condition and the uniqueness of viscosity solutions hold, then the game admits a value, i.e., the lower and upper value functionals coincide. This does not require the approximating technique of the (lower and upper) PHJI equations to the state dependent (not path-dependent) HJI equations studied in \cite{Zhang_SICON_2017}. The general uniqueness of viscosity solutions in our paper will be investigated in a future research study. Instead, we provide the uniqueness of classical solutions for the (lower and upper) state path-dependent HJI equations. In particular, under an additional assumption (see Assumption \ref{Assumption_3}), we prove the comparison principle of classical sub- and super-solutions of the lower and upper state path-dependent HJI equations, which further implies the uniqueness of classical solutions (for the state path-dependent case).

The paper is organized as follows. In Section \ref{Section_1_1}, we provide notation and preliminary results of the functional It\^o calculus introduced in \cite{Cont_JFA_2010, Cont_AOP_2013, Dupire_2009, Saporito_SICON_2019}. The problem formulation is given in Section \ref{Section_2}. In Section \ref{Section_3}, we show that the dynamic programming principle and obtain the regularity of the value functionals. In Section \ref{Section_4}, we introduce the lower and upper PHJI equations and prove that the value functionals are viscosity solutions of the corresponding PHJI equations. Several potential future research problems are also discussed at the end of Section \ref{Section_4}.

\section{Notation and Preliminaries}\label{Section_1_1}
The $n$-dimensional Euclidean space is denoted by $\mathbb{R}^n$, and the transpose of a vector $x \in \mathbb{R}^n$ by $x^\top$. The inner product of $x,y \in \mathbb{R}^n$ is denoted by $\langle x,y \rangle := x^\top y$, and the Euclidean norm of $x \in \mathbb{R}^n$ by $|x| := \langle x,x\rangle^{\frac{1}{2}}$. Let $\Tr(X)$ be the trace operator of a square matrix $X \in \mathbb{R}^{n \times n}$. Let $\mathds{1}$ be the indicator function. Let $\mathbb{S}^n$ be the set of $n\times n$ symmetric matrices.

We introduce the calculus of path-dependent functionals in \cite{Cont_JFA_2010, Cont_AOP_2013, Dupire_2009}; see also \cite{Ekren_AP_2014, Saporito_SICON_2019, Tang_DCD_2015}. We follow the notation in \cite{Dupire_2009, Saporito_SICON_2019}. For a fixed $T > 0$ and $t \in [0,T]$, let $\Lambda_t^n :=  C([0,t],\mathbb{R}^n)$ be the set of $\mathbb{R}^n$-valued continuous functions on $[0,t]$, and $\hat{\Lambda}_t^n :=  D([0,t],\mathbb{R}^n)$ the set of $\mathbb{R}^n$-valued c\`adl\`ag functions on $[0,t]$. Let $\Lambda_t^E := C([0,t],E)$ and $\hat{\Lambda}_t^E := D([0,t],E)$ for $E \subset \mathbb{R}^n$. Let $\Lambda_t^{n \times m} := \Lambda_t^n \times \Lambda_t^m$, $\Lambda^n := \cup_{t \in [0,T]} \Lambda_t^n$, and $\Lambda^{n \times m} := \cup_{t \in [0,T]} \Lambda_t^n \times \Lambda_t^m$. 
%When considering the delay system, we can extend the notation to $\Lambda^n = \cup_{t \in [-\tau,T]} \Lambda_t^n$, where $\tau > 0$ is a fixed delay. 
For any functions in $\Lambda^n$, the capital letter stands for the \emph{path} and the lowercase letter will denote the value of the function at a specific time. Specifically, for $A \in \Lambda_T^n$, $a_t$ stands for the value of $A$ at $t \in [0,T]$, and for $t \in [0,T]$, we denote $A_t := \{a_r,~ r \in [0,t]\} \in \Lambda_t^n$ by the path of the corresponding function up to time $t \in [0,T]$. A similar notation applies to $\hat{\Lambda}^n$. Note that $\Lambda^n \subset \hat{\Lambda}^n$.

For $A \in \hat{\Lambda}^n$ and $\delta > 0$, we introduce the following notation:
\begin{align*}
A_{t,\delta t}[s] := \begin{cases}
 	a_s & \text{if}~ s \in [0,t) \\
 	a_t & \text{if}~ s \in [t,t+\delta t]
 \end{cases},~ A_{t}^{(h)}[s] := \begin{cases}
 	a_s & \text{if}~ s \in [0,t) \\
 	a_t + h & \text{if}~ s = t.
 \end{cases} 
\end{align*}
Note that $A_{t,\delta t}$ is the flat extension, and $A_{t}^h$ is the vertical extension of the path $A$.  The metric on $\hat{\Lambda}^n$ is defined for $A_t,B_{t^\prime} \in \hat{\Lambda}^n$ with $t,t^\prime \in [0,T]$ and $t \leq t^\prime$,
\begin{align*}	
d_{\infty}(A_t,B_{t^\prime}) := |t-t^\prime| + \|A_{t,t^\prime - t} - B_{t^\prime} \|_{\infty},
\end{align*}
where $\| \cdot \|_{\infty}$ is the norm in $\hat{\Lambda}_t^n$ defined by $\|B_t\|_\infty := \sup_{r \in [0,t]} |b_r|$. 

Note that $(\hat{\Lambda}^n,d_{\infty})$ is a complete metric space, and $(\hat{\Lambda}_t^n,\|\cdot\|_{\infty})$ is a Banach space. The same results hold for $(\Lambda^n,d_{\infty})$ and $(\Lambda_t^n,\|\cdot\|_{\infty})$. Unfortunately, $\hat{\Lambda}_t^n$ is  not separable under the metric $\tilde{d}$ induced by $\|\cdot\|_{\infty}$. Therefore, we introduce the following metric: $d_{\infty}^\prime	(A_t,B_{t^\prime}) := |t-t^\prime| + d^\circ (A_{t,t^\prime - t},B_{t^\prime})$, where with $\Gamma$, the class of strictly increasing and continuous mappings, the \emph{Skorohod metric} \cite[Section 12, (12.13)]{Billingsley_book} is defined by $d^\circ (A_{t,t^\prime - t},B_{t^\prime}) := \inf_{\iota \in \Gamma} \{ \sup_{r \in [0,t^\prime]} |\iota_r - r| \vee \|A_{t,t^\prime - t} - B_{t^\prime \iota}\|_{\infty} \}$. By definition, $\|A_{t,t^\prime - t} - B_{t^\prime \iota}\|_{\infty} = \sup\{\sup_{r \in [0,t]} |a_r - b_{\iota_r}|, \sup_{r \in [t,t^\prime]} |a_t - b_{\iota_r}| \}$, which allows a deformation on the time scale to define a distance between $A$ and $B$. As shown in \cite[Section 12]{Billingsley_book}, $d^\circ$ is a metric and so is $d_{\infty}^\prime$. Then $\hat{\Lambda}_t^n$ is separable under $d^\circ$ \cite[Theorem 12.2]{Billingsley_book}. We can easily see that $d^\circ \leq \tilde{d}$, which implies $d_{\infty}^\prime \leq d_{\infty}$.

\begin{definition}\label{Definition_0_1}
A functional is any function $f : \hat{\Lambda}^n \rightarrow \mathbb{R}$. The functional $f$ is said to be continuous at $A_t \in \hat{\Lambda}^n$, if for each $\epsilon > 0$, there exists $\delta > 0$ such that for each $A_{t^\prime}^\prime \in \hat{\Lambda}^n$, $d_{\infty} (A_t,A_{t^\prime}^\prime) < \delta$ implies $|f(A_t) - f(A_{t^\prime}^\prime) |< \epsilon$. The continuity under $d_{\infty}$ implies the continuity under $d_{\infty}^\prime$. Let $\mathcal{C}(\hat{\Lambda}^n)$ be the set of real-valued continuous functionals for every path $A_t \in \hat{\Lambda}^n$ under $d_{\infty}$. The set $\mathcal{C}(\Lambda^n)$ is defined similarly.
%We write $f \in \mathcal{C}_b(\Lambda^n)$ if $f \in \mathcal{C}(\Lambda^n)$ and is bounded. We write $f \in \mathcal{C}_l(\Lambda^n)$ if $f \in \mathcal{C}(\Lambda^n)$ and satisfies the linear growth condition.
%
%We write $f \in \mathcal{C}(\Lambda^n)$ if $f: \Lambda^n \rightarrow \mathbb{R}$ is continuous at every path $A_t \in \Lambda^n$ under $d_{\infty}$.
\end{definition}

Next, we introduce the concept of time and space derivatives of the functional $f$.

\begin{definition}\label{Definition_0_2}
\begin{enumerate}[(i)]
	\item Let $f:\hat{\Lambda}^n \rightarrow \mathbb{R}$ be the functional. The time derivative (or horizontal derivative) of $f$ at $A_t$ is defined by $\partial_t f(A_t) := \lim_{\delta t \downarrow 0}	\frac{f(A_{t,\delta t} ) - f(A_t)}{\delta t}$. If the limit exists for all $A_t \in \hat{\Lambda}^n$, a functional $\partial_t f:\hat{\Lambda}^n \rightarrow \mathbb{R}$ is called the time derivative of $f$.
	\item 
%	Let $f:\hat{\Lambda}^n \rightarrow \mathbb{R}$ be the functional. 
	The space derivative (or vertical derivative) of $f$ at $A_t$ is defined by $\partial_x f(A_t) := \begin{bmatrix}
 			\partial_x^{1} f(A_t) & \cdots & \partial_x^{n} f(A_t) \end{bmatrix}$, where for $e_i$, $i=1,\ldots,n$, being a coordinate unit vector of $\mathbb{R}^n$, $\partial_x^{i} f(A_t) := \lim_{h \downarrow 0} \frac{f(A_t^{^{(h e_i)}}) - f(A_t)}{h}$. If the limit exists for all $A_t \in \hat{\Lambda}^n$ and $i=1,\ldots,n$, a functional $\partial_x f:\hat{\Lambda}^n \rightarrow \mathbb{R}^n$ is called the space derivative of $f$. Note that the second-order space derivative (Hessian) $\partial_{xx} f$ can be defined in a similar way, where $\partial_{xx} f : \hat{\Lambda}^n \rightarrow \mathbb{S}^n$.
\end{enumerate}	
\end{definition}

\begin{remark}\label{Remark_0_3}
	If a functional $f$ is differentiable in the sense of Definition \ref{Definition_0_2} and depends only on a function (not its path), i.e., $f(A_t) = f(t,a_t)$, then the notion of derivatives in Definition \ref{Definition_0_2} is equivalent to those for the classical ones. 
\end{remark} 

From Definition \ref{Definition_0_2}, let $\mathcal{C}^{k,l}(\hat{\Lambda}^n)$ be the set of functionals such that for $f \in \mathcal{C}^{k,l}(\hat{\Lambda}^n)$, $f$ is $k$ times time differentiable and $l$ times space differentiable in $\hat{\Lambda}^n$, where all its derivatives are continuous in the sense of Definition \ref{Definition_0_1}. The set $\mathcal{C}^{k,l}(\Lambda^n)$ is defined similarly. We mention that these sets are well defined in view of \cite[Definition 2.4 and Remark 2]{Tang_DCD_2015} (see also \cite[Theorem 2.4]{Ekren_AP_2014} and \cite{Cont_JFA_2010, Cont_AOP_2013, Dupire_2009}).
%We write $f \in \mathcal{C}_b^{k,l}(\Lambda^n)$ if $f \in \mathcal{C}^{k,l}(\Lambda^n)$ and all its derivatives are and bounded. Let $\mathcal{C}_l^{k,l}(\Lambda^n)$ the set of functionals such that $f \in \mathcal{C}_l^{k,l}(\Lambda^n)$ and all its derivatives satisfy the linear growth condition.

\begin{definition}\label{Definition_0_4}
Let $A_t \in \Lambda^n$. For any $\kappa \in (0,1]$, $A$ is an $\kappa$-H\"older continuous path if the following limit exists: $[\![ A_t ]\!]_{\kappa} :=\sup_{0 \leq s \leq r \leq t} \frac{|a_s - a_r|}{|s-r|^\kappa} < \infty$, where we call $[\![ A_t ]\!]_{\kappa}$ the $\kappa$-H\"older modulus of $A_t$. The $\kappa$-H\"older space is defined by $\mathbb{C}^{\kappa}(\Lambda^n) := \{ A_t \in \Lambda^n~:~	[\![ A_t ]\!]_{\kappa} < \infty \}$. The $\kappa$-H\"older space with $\mu > 0$ is defined by $\mathbb{C}^{\kappa,\mu}(\Lambda^n) := \{ A_t \in \Lambda^n~:~ 	[\![ A_t ]\!]_{\kappa} \leq \mu \}$. The $\kappa$-H\"older space with $\mu > 0$ and $\mu_0 > 0$ is defined by 
\begin{align*}
	\mathbb{C}^{\kappa,\mu, \mu_0}(\Lambda^n) := \{ A_t \in \Lambda^n~:~ [\![ A_t ]\!]_{\kappa} \leq \mu,~ \|A_t\|_{\infty} \leq \mu_0 \}.
\end{align*}
\end{definition}

We can easily see that $\mathbb{C}^{\kappa}(\Lambda^n) \subset \Lambda^n$. The spaces $\mathbb{C}^{\kappa,\mu}(\Lambda^n)$ and $\mathbb{C}^{\kappa,\mu, \mu_0}(\Lambda^n)$ have the following topological property \cite[Proposition 1]{Tang_DCD_2015}:

\begin{lemma}\label{Lemma_0_5}
	For $\kappa \in (0,1]$, $\mathbb{C}^{\kappa,\mu}(\Lambda^n)$ and $\mathbb{C}^{\kappa,\mu, \mu_0}(\Lambda^n)$ are compact subsets of $(\Lambda^n,d_{\infty})$.
\end{lemma}

%\begin{definition}	\label{Definition_0_3}
%Let $A \in \Lambda^n$. For any $\kappa \in (0,1]$, $A$ is a $\kappa$-H\"older continuous path if the following limit exists
%\begin{align*}
%[\![ A_t ]\!]_{\kappa} :=\sup_{0 \leq s \leq r \leq t} \frac{|a_s - a_r|}{|s-r|^\kappa} < \infty.
%\end{align*}
%We call $[\![ A_t ]\!]_{\kappa}$ the $\kappa$-H\"older modulus of $A_t$. The $\kappa$-H\"older space is defined by
%\begin{align*}
%\mathbb{C}^{\kappa}(\Lambda^n) := \{ A_t \in \Lambda^n~:~ 	[\![ A_t ]\!]_{\kappa} < \infty \}.
%\end{align*}
%The $\kappa$-H\"older space with $\mu > 0$ is defined by
%\begin{align*}	
%\mathbb{C}^{\kappa,\mu}(\Lambda^n) := \{ A_t \in \Lambda^n~:~ 	[\![ A_t ]\!]_{\kappa} \leq \mu \}.
%\end{align*}
%\end{definition}
%
%We can easily see that $\mathbb{C}^{\kappa}(\Lambda^n) \subset \Lambda^n$. The space $\mathbb{C}^{\kappa,\mu}(\Lambda^n)$ has the following topological property:
%\begin{lemma}
%	For $\kappa \in (0,1]$, $\mathbb{C}^{\kappa,\mu}(\Lambda^n)$ is a compact subset of $(\Lambda^n,d_{\infty})$.
%\end{lemma}
%
%
%
%Similar to Definition \ref{Definition_0_4}, we introduce the notion of the H\"older continuity for the functional $f$.
\begin{definition}\label{Definition_0_6}
Let $f : \Lambda^n \rightarrow \mathbb{R}$ be the functional. For $\kappa \in (0,1]$, $f$ is H\"older continuous if  the following limit exists: $[f]_{\kappa; \Lambda^n} := \sup_{A_t,A_{t^\prime}^\prime \in \Lambda^n, A_t \neq A_{t^\prime}^\prime  } \frac{|f(A_t) - f(A_{t^\prime}^\prime)|}{d_{\infty}^\kappa (A_t,A_{t^\prime}^\prime)}$. Assume that $f \in \mathcal{C}^{1,2}(\Lambda^n)$. We define $|f|_{\kappa;\Lambda^n} := \sup_{A_t \in \Lambda^n} |f(A_t)| + [f]_{\kappa;\Lambda^n} \nonumber$ and
\begin{align}	
%|f|_{\kappa;\Lambda^n} &:= \sup_{A_t \in \Lambda^n} |f(A_t)| + [f]_{\kappa;\Lambda^n} \nonumber  \\
\label{eq_0_1}
|f|_{2,\kappa;\Lambda^n} &:= |f|_{\kappa;\Lambda^n} + |\partial_t f|_{\alpha;\Lambda^n} + |\partial_x f|_{\kappa;\Lambda^n} + |\partial_{xx} f|_{\kappa;\Lambda^n}.
\end{align}
The set of functionals such that (\ref{eq_0_1}) is finite is denoted by $\mathcal{C}^{1,2}_{\kappa} (\Lambda^n)$.
\end{definition}

Let $(\Omega,\mathcal{F},\mathbb{P})$ be a complete probability space satisfying the usual condition \cite{Karantzas}. Let $B$ be the standard $p$-dimensional Brownian motion defined on $(\Omega,\mathcal{F},\mathbb{P})$. Let $\mathbb{F} = \{ \mathcal{F}_t,~ 0 \leq t \leq T\}$ be the standard natural filtration generated by the Brownian motion $B$ augmented by all the $\mathbb{P}$-null sets of $\mathcal{F}$. Let $\mathcal{L}^2(\Omega,\mathcal{F}_t,\mathbb{R}^n)$ be the set of $\mathbb{R}^n$-valued $\mathcal{F}_t$-measurable random vectors such that $g \in \mathcal{L}^2(\Omega,\mathcal{F}_t,\mathbb{R}^n)$ satisfies $\mathbb{E}[|g|^2] < \infty$. Let $\mathcal{L}^2_{\mathbb{F}}([t,T],\mathbb{R}^n)$ be the set of $\mathbb{R}^n$-valued $\mathbb{F}$-adapted stochastic processes such that $g \in \mathcal{L}^2_{\mathbb{F}}([t,T],\mathbb{R}^n)$ satisfies $\mathbb{E}[\int_t^T |g(s)|^2 \dd s] < \infty$. Let $\mathcal{C}_{\mathbb{F}}([t,T],\mathbb{R}^n)$ be the set of $\mathbb{R}^n$-valued continuous and $\mathbb{F}$-adapted stochastic processes such that $g \in \mathcal{C}_{\mathbb{F}}([t,T],\mathbb{R}^n)$ satisfies $\mathbb{E}[\sup_{s \in [t,T]} |g(s)|^2 ] < \infty$. 

Let $x \in \Lambda_T^n$ be the $n$-dimensional $\mathbb{F}$-adapted stochastic process, which is defined on $(\Omega,\mathcal{F},\mathbb{P})$. Note that $x$ can be viewed as a mapping from $\Omega$ to $\Lambda_T^n$. By using the notation, for $ t \in [0,T]$, $X_t := \{x_r,~ r \in [0,t]\} \in \Lambda_t^n$ is the path of $x$ up to time $t \in [0,T]$, and $x_t$ is the value of $X$ at time $t \in [0,T]$. We can see that for any functional $f \in \mathcal{C}(\Lambda^n)$, $\{f(X_t),~ t \in [0,T]\}$ is an $\mathbb{F}$-adapted stochastic process.

We now state the functional It\^o formula in \cite{Cont_JFA_2010, Cont_AOP_2013, Dupire_2009}
\begin{lemma}\label{Lemma_0_7}
	Suppose that $x$ is continuous semi-martingale, and $f \in \mathcal{C}^{1,2}(\Lambda^n)$. Then for any $t \in [0,T]$, 
	\begin{align*}
	f(X_t) &= f(X_0)  + \int_0^t \partial_t f(X_r) \dd r + \int_0^t \partial_x f(X_r) \dd  x_r  + \frac{1}{2} \int_0^t \partial_{xx} f(X_r) \dd \langle x \rangle_r,~ \text{$\mathbb{P}$-a.s.}
	\end{align*}

\end{lemma}

\section{Problem Formulation}\label{Section_2}

This section provides the precise problem formulation of state and control path-dependent SZSDGs. The state and control path-dependent problem was first introduced in \cite{Saporito_SICON_2019} to solve the stochastic control problem and (nonzero-sum) differential game. On the other hand, we study the (state and control path-dependent) problem in the SZSDG framework.

%For $ t \in [0,T)$, let $u:[t,T] \times \Omega \rightarrow \mathrm{U}$, where $\mathrm{U} \subset \mathbb{R}^{m}$, be the control process of Player 1. The control process of Player 2 is defined similarly with $\mathrm{V} \subset \mathbb{R}^{l}$. It is assumed that $\mathrm{U}$ and $\mathrm{V}$ are compact metric spaces with the standard Euclidean norm. 

Let $\mathcal{U}$ be the set of $\mathrm{U}$-valued $\mathcal{F}$-progressively measurable and c\`adl\`ag processes, where $\mathrm{U} \subset \mathbb{R}^{m}$, which is the set of control processes for Player 1. The set of control processes for Player 2, $\mathcal{V}$, is defined similarly with $\mathrm{V} \subset \mathbb{R}^{l}$. It is assumed that $\mathrm{U}$ and $\mathrm{V}$ are compact metric spaces with the standard Euclidean norm. The precise definitions of $\mathcal{U}$ and $\mathcal{V}$ are given later.

The state and control path-dependent stochastic differential equation (SDE) is given by
\begin{align}
\label{eq_2_1}
\begin{cases}
	\dd x_s^{t,A_t;U,V} = f(X_s^{t,A_t;U,V},U_s,V_s)\dd s + \sigma(X_s^{t,A_t;U,V},U_s,V_s)\dd B_s,~ s \in (t,T] \\
	X_t^{t,A_t;U,V} = A_t \in \Lambda_t^n,
\end{cases}
\end{align}
where $X_s^{t,A_t;U,V} := \{x_r^{t,A_t;U,V} \in \mathbb{R}^n,~ r \in [0,s] \} \in \Lambda_s^n$ is the whole path of the controlled state process from time $0$ to $s$, and $U_s := \{u_r \in  \mathrm{U},~ r \in [0,s]; ~u \in \mathcal{U}\} \in \hat{\Lambda}_s^{\mathrm{U}} \subset \hat{\Lambda}_s^m$ and $V_s := \{v_r \in  \mathrm{V},~ r \in [0,s];~ v \in \mathcal{V}\} \in \hat{\Lambda}_s^{\mathrm{V}} \subset \hat{\Lambda}_s^l $ are paths of the control processes of Players 1 and 2, respectively. In (\ref{eq_2_1}), $A_t \in \Lambda_t^n$ is the initial condition that is a continuous path starting from time $t=0$. Let $\Lambda := \Lambda^n$ and $\hat{\Lambda} := \hat{\Lambda}^{\mathrm{U}} \times \hat{\Lambda}^{\mathrm{V}}$.

The state and control path-dependent backward stochastic differential equation (BSDE) is given by
\begin{align}
\label{eq_2_2}
\begin{cases}
\dd y_s^{t,A_t;U,V} = - l(X_s^{t,A_t;U,V},y_s^{t,A_t;U,V},q_s^{t,A_t;U,V}, U_s,V_s)\dd s \\
~~~~~~~~~~~~~~~~ + q_s^{t,A_t;U,V} \dd B_s,~ s \in [t,T)	\\
y_T^{t,A_t;U,V} = m(X_T^{t,A_t;U,V}),
\end{cases}
\end{align}
where the pair $(y_s^{t,A_t;U,V},q_s^{t,A_t;U,V}) \in \mathbb{R} \times \mathbb{R}^{1 \times p}$ is the solution of the BSDE. Note that the BSDE in (\ref{eq_2_2}) is coupled with the (forward) SDE in (\ref{eq_2_1}). Below, the BSDE in (\ref{eq_2_2}) is used to define the objective functional of Players 1  and 2.

We introduce the following assumption:
\begin{assumption}\label{Assumption_1}
In (\ref{eq_2_1}), the coefficients $f:\Lambda \times \hat{\Lambda}  \rightarrow \mathbb{R}^n$ and $\sigma: \Lambda \times \hat{\Lambda} \rightarrow \mathbb{R}^{n \times p}$ are bounded. Furthermore, the running and terminal costs in (\ref{eq_2_2}), $l:\Lambda \times \mathbb{R} \times \mathbb{R}^{1 \times p} \times \hat{\Lambda} \rightarrow \mathbb{R}$ and $m:\Lambda_T \rightarrow \mathbb{R}$, respectively, are bounded. There exists a constant $L > 0$ such that for $s_i \in [0,T]$ and $(X_T^i, U_T^i, V_T^i, y_i, q_i) \in \Lambda_T \times \hat{\Lambda}_T \times \mathbb{R} \times \mathbb{R}^{1 \times p}$, $i=1,2$, the following conditions hold: 
\begin{align*}	
& |f(X_{s_1}^{1},U_{s_1}^{1},V_{s_1}^{1}) - f(X_{s_2}^{2},U_{s_2}^{2},V_{s_2}^{2})| \\
& \leq L (d_{\infty}(X_{s_1}^{1},X_{s_2}^{2}) + d_{\infty}(U_{s_1}^{1}, U_{s_2}^{2}) + d_{\infty}(V_{s_1}^{1}, V_{s_2}^2)) \\
& |\sigma(X_{s_1}^{1},U_{s_1}^{1},V_{s_1}^{1}) - \sigma(X_{s_2}^{2},U_{s_2}^{2},V_{s_2}^{2})| \\
& \leq L (d_{\infty}(X_{s_1}^{1},X_{s_2}^{2}) + d_{\infty}(U_{s_1}^{1}, U_{s_2}^{2}) + d_{\infty}(V_{s_1}^{1}, V_{s_2}^2)) \\
& | l(X_{s_1}^{1},y_1,q_1,U_{s_1}^{1},V_{s_1}^{1}) - l(X_{s_2}^{2},y_2,q_2,U_{s_2}^{2},V_{s_2}^{2})|  \\
& \leq L (d_{\infty}(X_{s_1}^{1},X_{s_2}^{2}) + d_{\infty}(U_{s_1}^{1},U_{s_2}^{2}) + d_{\infty} (V_{s_1}^{1},V_{s_2}^{2}) + |y_1 - y_2| + |q_1 - q_2|) \\
& |m(X_T^{1}) - m(X_T^{2}) | \leq L \|X_T^{1} - X_T^{2}\|_{\infty}.
\end{align*}
\end{assumption}

Based on \cite{Buckdahn_SICON_2008, Li_SICON_2014, Tang_DCD_2015, Touzi_Book, Zhang_SICON_2017, Zhang_book_2017}, we have the following result:
\begin{lemma}\label{Lemma_1}
Suppose that Assumption \ref{Assumption_1} holds. Then, the following hold:
\begin{enumerate}[(i)]
\item For $t \in [0,T)$, $A_t \in \Lambda_t$ and $(U,V) \in \hat{\Lambda}$, the SDE in (\ref{eq_2_1}) and the BSDE in (\ref{eq_2_2}) admit unique strong solutions, $X^{t,A_t;U,V} $ with $\mathbb{E}[\|X_T^{t,A_t;U,V}\|^2_{\infty} | \mathcal{F}_t ] < \infty$ and $(y^{t,A_t;U,V},z^{t,A_t;U,V}) \in \mathcal{C}_{\mathbb{F}}([t,T],\mathbb{R}^n) \times \mathcal{L}_{\mathbb{F}}^2([t,T],\mathbb{R}^{1 \times p})$, respectively.
\item For, $t \in [0,T)$, $t_1,t_2 \in [t,T]$ with $t_2 \geq t_1$, $A_t^{i} \in \Lambda_t$, and $(U^i, V^i) \in \hat{\Lambda}$, $i=1,2$, there exists a constant $C > 0$, dependent on the Lipschitz constant $L$ in Assumption \ref{Assumption_1}, such that
\begin{align*}
& \mathbb{E} \Bigl [ \| X_T^{t,A_t^{1}; U^1,V^1} \|_\infty^2 \bigl | \mathcal{F}_t \Bigr ] \leq C (1 + \| A_t^1 \|_{\infty}^2 ) \\
& \mathbb{E} \Bigl [ \| X_{t_2}^{t,A_t^{1}; U^1,V^1} - A_{t_1,t_2-t_1} \|_\infty^2 \bigl | \mathcal{F}_t \Bigr ] \leq C (1 + \| A_{t_1}^1 \|_{\infty}^2 ) (t_2 - t_1)\\
& \mathbb{E} \Bigl [ \| X_T^{t,A_t^{1}; U^1,V^1} - 	X_T^{t,A_t^{2}; U^2,V^2} \|_\infty^2 \bigl | \mathcal{F}_t \Bigr ]   \\
&~~~ \leq  C \|A_t^{1} - A_t^{2} \|_{\infty} + C \mathbb{E} \Bigl [ \int_t^T \bigl [ \| U_r^{1} - U_r^{2} \|_{\infty}^2 + \| V_r^{1} - V_r^{2} \|_{\infty}^2  \bigr ] \dd r \bigl | \mathcal{F}_t \Bigr ].
\end{align*}
\item For, $t \in [0,T)$, $t_1,t_2 \in [t,T]$ with $t_2 \geq t_1$, $A_t^{i} \in \Lambda_t$, and $(U^i, V^i) \in \hat{\Lambda}$, $i=1,2$, there exists a constant $C > 0$, dependent on the Lipschitz constant $L$ in Assumption \ref{Assumption_1}, such that
\begin{align*}	
& \mathbb{E} \Bigl [ \sup_{s \in [t,T]} |y_s^{t,A_t^1,U^1,V^1} |^2   + \int_t^T |q_r^{t,A_t^1,U^1,V^1} |^2  \dd r \bigl |  \mathcal{F}_t \Bigr ] \leq C (1 + \|A_t^1\|_{\infty}^2 )
\\
& \mathbb{E} \Bigl [ \sup_{s \in [t_1,t_2]} |y_s^{t_1,A_{t_1}^1,U^1,V^1} - y_{t_1}^{t_1,A_{t_1}^1,U^1,V^1} |^2 \bigl | \mathcal{F}_t \Bigr ] \leq C (1 + \|A_t^1\|_{\infty}^2 )(t_2 - t_1)
\\
& \mathbb{E} \Bigl [ \sup_{s \in [t,T]} |y_s^{t,A_t^1,U^1,V^1} - y_s^{t,A_t^2,U^2,V^2} |^2   \bigl | \mathcal{F}_t \Bigr ]\\
&~~~ \leq C \| A_t^1 - A_t^2 \|_{\infty}^2 + C \mathbb{E} \Bigl [ \int_t^T \bigl [ \| U_r^{1} - U_r^{2} \|_{\infty}^2 + \| V_r^{1} - V_r^{2} \|_{\infty}^2  \bigr ] \dd r \bigl | \mathcal{F}_t \Bigr ].
\end{align*}
\item Suppose that $l^{(1)}$ and $l^{(2)}$ are coefficients of the BSDE in (\ref{eq_2_2}) satisfying Assumption \ref{Assumption_1}, and $\eta^{(1)}, \eta^{(2)} \in \mathcal{L}^2(\Omega,\mathcal{F}_T,\mathbb{R})$ are the corresponding terminal conditions. Let $(y^{(1)}, q^{(1)})$ and $(y^{(2)}, q^{(2)})$ be solutions of the BSDE in (\ref{eq_2_2}) with $(l^{(1)}, \eta^{(1)})$ and $(l^{(2)}, \eta^{(2)})$, respectively (note that $y_T^{(1)} = \eta^{(1)}$ and $y_T^{(2)} = \eta^{(2)}$). If $\eta^{(1)} \geq \eta^{(2)}$ and $l^{(1)} \geq l^{(2)}$, then $y_s^{(1)} \geq y_s^{(2)}$, a.s., for $ s \in [t,T]$.
\end{enumerate}
\end{lemma}

The objective functional of Players 1 and 2 is given by
\begin{align}
\label{eq_2_3}
J(t,A_t;U,V) = 	y_t^{t,A_t;U,V},~ t \in [0,T],
\end{align}
where $y$ is the first component of the BSDE in (\ref{eq_2_2}). Note that $J(T,A_t;U,V) = y_T^{t,A_t;U,V} = m(X_T^{t,A_t;U,V})$.

\begin{remark}
When $l$ in (\ref{eq_2_2}) is independent of $y$ and $q$, (\ref{eq_2_3}) becomes
\begin{align*}
J_t(t,A_t;U,V) & = \mathbb{E} \Bigl [ \int_t^T l(X_s^{t,A_t;U,V}, U_s,V_s)\dd s + m(X_T^{t,A_t;U,V}) \bigl | \mathcal{F}_t \Bigr ].
\end{align*}	
\end{remark}

The admissible control of Players 1 and 2 is defined as follows:
\begin{definition}\label{Definition_1}
For $t \in [0,T]$, the admissible control for Player 1 (respectively, Player 2) is defined such that $u := \{u_r \in \mathrm{U},~ r \in [t,T]\}$ (respectively, $v := \{v_r \in \mathrm{V},~ r \in [t,T]\}$) is an $\mathrm{U}$-valued (respectively, $\mathrm{V}$-valued) $\mathbb{F}$-progressively measurable and c\`adl\`ag process in $\mathcal{L}_{\mathbb{F}}^2([t,T],\mathrm{U})$ (respectively, $\mathcal{L}_{\mathbb{F}}^2([t,T],\mathrm{V})$). The set of admissible controls of Player 1 (respectively, Player 2) is denoted by $\mathcal{U}[t,T]$ (respectively, $\mathcal{V}[t,T]$). We identify two admissible control processes of Player 1 (respectively, Player 2) $u$ and $\bar{u}$ in $\mathcal{U}[t,T]$ (respectively, $v$ and $\bar{v}$ in $\mathcal{V}[t,T]$) and write $u \equiv \bar{u}$ (respectively, $v \equiv \bar{v}$) on $[t,T]$, if $\mathbb{P}(u = \bar{u}~ \text{a.e}~ \text{in}~ [t,T] ) = 1$ (respectively, $\mathbb{P} (v = \bar{v}~ \text{a.e}~ \text{in}~ [t,T] ) = 1$). 
\end{definition}

Given the definition of the admissible controls for Players 1 and 2, we introduce the concept of \emph{nonanticipative strategies} for Players 1 and 2.

\begin{definition}\label{Definition_2}
For $t \in [0,T]$, a nonanticipative strategy for Player 1 (respectively, Player 2) is a 
%c\`adl\`ag 
mapping $\alpha:\mathcal{V}[t,T] \rightarrow \mathcal{U}[t,T]$ (respectively, $\beta:\mathcal{V}[t,T] \rightarrow \mathcal{U}[t,T]$) such that for any $\mathbb{F}$-stopping time $\tau:\Omega \rightarrow [t,T]$ and any $u^{1}, u^{2} \in \mathcal{U}$ with $u^{1} \equiv u^{2}$ on $[t,\tau]$ (respectively, $v^{1}, v^{2} \in \mathcal{V}$ with $v^{1} \equiv v^{2}$ on $[t,\tau]$), it holds that $\alpha(u^{1}) \equiv \alpha(u^{2})$ on $[t,\tau]$ (respectively, $\beta(u^{1}) \equiv \beta(u^{2})$ on $[t,\tau]$). The set of admissible strategies for Player 1 (respectively, Player 2) is denoted by $\mathcal{A}[t,T]$ (respectively, $\mathcal{B}[t,T]$).  
\end{definition}

The following notation captures control path-dependent SZSDGs: for $t \in [0,T)$,
\begin{align}
\label{eq_2_4}
(Z_{t} \otimes u)[s] & := 
\begin{cases}
 	z_s, &\text{if}~ s  \in [0,t) \\
 	u_s, & \text{if}~ s \in [t,T], \\
 \end{cases}~~~
 (W_{t} \otimes v)[s] := 
\begin{cases}
 	w_s, &\text{if}~ s  \in [0,t) \\
 	v_s, & \text{if}~ s \in [t,T], \\
 \end{cases}	
\end{align}
where $Z_t := \{z_r,~ r \in [0,t] \}\in \hat{\Lambda}_t^{\mathrm{U}}$, $u \in \mathcal{U}[t,T]$, $W_t := \{w_r,~ r \in [0,t] \} \in \hat{\Lambda}_t^{\mathrm{V}}$, and $v \in \mathcal{V}[t,T]$. Note that 
%$(Z_{t} \otimes u)$ and $(W_{t} \otimes v)$ are admissible, i.e., 
$(Z_{t} \otimes u) \in \mathcal{U}[0,T]$ and $(W_{t} \otimes v) \in \mathcal{V}[0,T]$.

With the help of the notation in (\ref{eq_2_4}), the objective functional of (\ref{eq_2_3}) that includes the \emph{path} of the control of Players 1 and 2 can be written as follows:
\begin{align}
\label{eq_2_5}
J(t,A_t; Z_{t} \otimes u, W_{t} \otimes v) = y_t^{t,A_t; Z_{t} \otimes u, W_{t} \otimes v}.
\end{align}
Then for $(t,A_t) \in [0,T] \times \Lambda_t$ and $(Z_t,W_t) \in \hat{\Lambda}_t$, the lower value functional of (\ref{eq_2_5}) for the state and control path-dependent SZSDG can be defined by
\begin{align}
\label{eq_2_6}
\mathbb{L}(A_t; Z_t, W_t) &= \essinf_{\alpha \in \mathcal{A}[t,T]} \esssup_{v \in \mathcal{V}[t,T]} 	J(t,A_t; Z_{t} \otimes \alpha (W_t \otimes v), W_t \otimes v) \\
& = \essinf_{\alpha \in \mathcal{A}[t,T]} \esssup_{v \in \mathcal{V}[t,T]} 	J(t,A_t; Z_{t} \otimes \alpha (v), W_t \otimes v), \nonumber
\end{align}
where the last equality follows from (\ref{eq_2_4}). 
%This means that $\alpha (W_t \otimes v) = \alpha (v)$ is still nonanticipative. 
%We can easily see that $Z_{t} \otimes \alpha (v) \in \mathcal{A}[0,T]$ and $W_t \otimes v \in \mathcal{V}[0,T]$. 
Moreover, for $(t,A_t) \in [0,T] \times \Lambda_t$ and $(Z_t,W_t) \in \hat{\Lambda}_t$, the upper value functional  of (\ref{eq_2_5}) is defined by
\begin{align}
\label{eq_2_7}	
\mathbb{U}(A_t; Z_t, W_t) &=  \esssup_{\beta \in \mathcal{B}[t,T]} \essinf_{u \in \mathcal{U}[t,T]} 	J(t,A_t; Z_{t} \otimes u, W_t \otimes \beta (Z_{t} \otimes u)) \\
& = \esssup_{\beta \in \mathcal{B}[t,T]} \essinf_{u \in \mathcal{U}[t,T]} 	J(t,A_t; Z_{t} \otimes u, W_t \otimes \beta (u)). \nonumber
\end{align}
Note that $\mathbb{L}(A_T; Z_T, W_T) = \mathbb{U}(A_T; Z_T, W_T) = m(A_T)$.

We state some remarks on various formulations of (path-dependent) SZSDGs.

\begin{remark}\label{Remark_3_6}
\begin{enumerate}[(1)]
\item One might formulate SZSDGs with control against control, in which the players can select admissible controls individually. Although this formulation is quite similar to stochastic optimal control and therefore can define the saddle-point equilibrium, the dynamic programming principle cannot be established and the value of the game may fail to exist; see \cite[Appendix E]{Pham_SICON_2014} and \cite[Example 2.1]{Possamai_arXiv_2018}. Note that under this formulation, the necessary condition for the existence of the saddle-point equilibrium in terms of the (stochastic) maximum principle was studied in \cite{Moon_TAC_Risk_2019}.
\item The notion of nonanticipative strategies in Definition \ref{Definition_2} is used in various zero-sum differential games; see \cite{Basar2, Bayraktar_2013_SICON, Buckdahn_Cardaliaguet_DGAA_2011, Buckdahn_SICON_2004, Buckdahn_SICON_2008, Fleming_1989,  Li_AMO_2015, YU_SICON_2015, Zhang_SICON_2017}. This is the strong formulation with strategy against control. Under this formulation, it is possible to  establish the dynamic programming principle, to show the existence of viscosity solutions of Hamilton-Jacobi-Isaacs (HJI) equations, and to identify the existence of the game value under the Isaacs condition. We also note that instead of the strong formulation with strategy against control, one can use the notion of \emph{nonanticipative strategy with delay}, which is still asymmetric information between the players that allows to show the existence of the (approximated) saddle-point equilibrium and the game value \cite{Buckdahn_SICON_2004, Cardaliaguet_IGTR_2008, Zhang_SICON_2017}.
\item Instead of the strong formulation with strategy against control, SZSDGs can be considered in weak formulation \cite{Karoui_2003, Hamadene_SSR_1995, Li_MIn_SICON_2016, Pham_SICON_2014, Possamai_arXiv_2018}. Note that in \cite{Pham_SICON_2014, Possamai_arXiv_2018}, the players are restricted to observe the state feedback information. Since the information is symmetric, it is convenient to define the saddle-point equilibrium and show the existence of the game value. The dynamic programming principle can also be obtained. Note that the notion of viscosity solutions of the HJI equation requires the nonlinear expectation and some additional assumptions are required to show the existence and uniqueness of viscosity solutions in the sense of \cite{Ekren_AP_2014, Ekren_AP_2016_Part_I, Ekren_AP_2016}; see \cite[Remark 3.7 and p. 10]{Possamai_arXiv_2018} and \cite[Section 6]{Pham_SICON_2014}.
\end{enumerate}
\end{remark}

The next remark is on the (lower and upper) value functionals.
\begin{remark}\label{Remark_3_7}
%\begin{enumerate}[(1)]
%\item 
We can see that the value functionals in (\ref{eq_2_6}) and (\ref{eq_2_7}) depend on initial path of both state and control of the players. Consider the situation when the path-dependence is only in the state variable, i.e., $f:\Lambda \times \mathrm{U} \times \mathrm{V} \rightarrow \mathbb{R}^n$, $\sigma :\Lambda \times \mathrm{U} \times \mathrm{V} \rightarrow \mathbb{R}^{n \times p}$, and $l :\Lambda \times \mathbb{R} \times \mathbb{R}^{1 \times p} \times \mathrm{U} \times \mathrm{V} \rightarrow \mathbb{R} $. Then, the value functionals can be written independent of $Z$ and $W$:
\begin{align*}
\mathbb{L}(A_t) & = \essinf_{\alpha \in \mathcal{A}[t,T]} \esssup_{v \in \mathcal{V}[t,T]} 	J(t,A_t; \alpha (v), v) \\
\mathbb{U}(A_t) &= \esssup_{\beta \in \mathcal{B}[t,T]} \essinf_{u \in \mathcal{U}[t,T]} 	J(t,A_t; u, \beta (u)).
\end{align*}
This is a special case of the SZSDG in this paper, which was studied in \cite{Zhang_SICON_2017}. In addition, for the state and control path-independent case, i.e., the SZSDG in the Markovian formulation, the value functionals are reduced by
\begin{align*}
\mathbb{L}(t,x) & = \essinf_{\alpha \in \mathcal{A}[t,T]} \esssup_{v \in \mathcal{V}[t,T]} 	J(t,x; \alpha (v), v) \\
\mathbb{U}(t,x) &= \esssup_{\beta \in \mathcal{B}[t,T]} \essinf_{u \in \mathcal{U}[t,T]} 	J(t,x; u, \beta (u)),
\end{align*}
for any initial state $x \in \mathbb{R}^n$ and $t \in [0,T]$; see \cite{Buckdahn_SICON_2008, Fleming_1989} and the references therein.
%\item For the value functionals in (\ref{eq_2_6}) and (\ref{eq_2_7}), the essential supremum and the essential infimum are taken with respect to indexed family of random variables; see the precise notion in \cite[Appendix A]{Karatzas_book_1998}.
%\end{enumerate}
\end{remark}

\begin{remark}
As mentioned in \cite{Saporito_SICON_2019}, one main example of the SZSDG considered in this paper is the delay problem with delay $r > 0$, where the SDE is given by
\begin{align*}
\begin{cases}
\dd {x}_s^{t,\bar{x};u,v} = \bigl [ f_1(x_s^{t,\bar{x};u,v}) + f_2(u_s, u_{s-r},v_s, v_{s-r}) \bigr ] \dd s + \sigma(x_s^{t,\bar{x};u,v}) \dd B_s,~ s \in (t,T] \\
x_s^{t,\bar{x};u,v} = \bar{x},~u_s = \bar{u}_s,~ v_s = \bar{v}_s,~ s \in [t-r,t],
\end{cases}
\end{align*}
and the objective functional is
\begin{align*}
\begin{cases}
\dd y_s^{t,\bar{x};u,v} = - \bigl [ l_1(x_s^{t,\bar{x};u,v},y_s^{t,\bar{x};u,v},q_s^{t,\bar{x};u,v}) + l_2(u_s, u_{s-r},v_s, v_{s-r}) \bigr ] \dd s \\
~~~~~~~~~~~~~~~~ + q_s^{t,\bar{x};u,v} \dd B_s,~ s \in [t,T)	\\
y_T^{t,\bar{x}_t;u,v} = m(x_T^{t,\bar{x};u,v}).
\end{cases}	
\end{align*}
Note the initial path of $u$ and $v$. Stochastic control problems and differential games with delay can be solved by infinite-dimensional approaches \cite{Chen_2015, Gozzi_SICON_2017_1, Gozzi_SICON_2017}.\footnote{Note that \cite{Gozzi_SICON_2017_1, Gozzi_SICON_2017} considered the one-player stochastic control problem with delay. Of course, it is interesting to study the approach of \cite{Gozzi_SICON_2017_1, Gozzi_SICON_2017} in the SZSDG framework.} However, their approaches are applicable only to the delay-type problem and cannot be used to solve the general path-dependent problem. We also note that the path-dependent problem can be converted into the stochastic control problem (or differential game) with \emph{random coefficients}; see \cite[Example 4.5]{Tang_DCD_2015} and \cite{Qui_SICON_2018}. More applications of control problems  and differential games with delay and random coefficients can be found in \cite{Carmona_JOTA_2018, Fabbri_Book, Qui_SICON_2018, Moon_TAC_2020_Random} and the references therein, where their problems can be handled by the path-dependent analysis studied in this paper.
\end{remark}

\section{Dynamic Programming Principle}\label{Section_3}

This section establishes the dynamic programming principle for the lower and upper value functionals. % in (\ref{eq_2_6}) and (\ref{eq_2_7}). 

We first state properties of the value functionals. The proof of the following result is similar to that in \cite[Proposition 3.3]{Buckdahn_SICON_2008} and \cite[Proposition 3.3]{Li_SICON_2014}.

\begin{proposition}\label{Proposition_1}
Assume that Assumption \ref{Assumption_1} holds. The lower and upper value functionals $\mathbb{L}$ and $\mathbb{U}$ in (\ref{eq_2_6}) and (\ref{eq_2_7}), respectively, are $\mathcal{F}_t$-measurable random variables. In fact, they are deterministic functionals on  $\Lambda \times \hat{\Lambda}$.
\end{proposition}

In view of Assumption \ref{Assumption_1} the estimates in Lemma \ref{Lemma_1}, and (\ref{eq_2_4}), the following result holds:

\begin{lemma}\label{Lemma_2}
Suppose that Assumption \ref{Assumption_1} holds. For any $t \in [0,T]$, $A_t^i \in \Lambda_t$ and $(Z_t^i, W_t^i) \in \hat{\Lambda}_t$, $i=1,2$, there exists a constant $C > 0$ such that the following estimates hold: for $\mathbb{G} := \mathbb{L},\mathbb{U}$, $\bigl |\mathbb{G}(A_t^{1}; Z_t^{1}, W_t^{1}) \bigr | \leq C(1 + \|A_t^{1}\|_{\infty})$ and 
\begin{align*}
%& \bigl |\mathbb{G}(A_t^{1}; Z_t^{1}, W_t^{1}) \bigr | \leq C(1 + \|A_t^{1}\|_{\infty})
%\\
& \bigl |\mathbb{G}(A_t^{1}; Z_t^{1}, W_t^{1})	- \mathbb{G}(A_t^{2}; Z_t^{2}, W_t^{2}) \bigr |  \\
& \leq C ( \|A_t^1  - A_t^2 \|_{\infty} + \|Z_t^1  - Z_t^2 \| _{\infty} + \|W_t^1  - W_t^2 \|_{\infty}).
%& \bigl |\mathbb{L}(A_t^{1}; Z_t^{1}, W_t^{1}) \bigr | \leq C(1 + \|A_t^{1}\|_{\infty})
%\\
%& \bigl |\mathbb{L}(A_t^{1}; Z_t^{1}, W_t^{1})	- \mathbb{L}(A_t^{2}; Z_t^{2}, W_t^{2}) \bigr |  \\
%& \leq C ( \|A_t^1  - A_t^2 \|_{\infty} + \|Z_t^1  - Z_t^2 \| _{\infty} + \|W_t^1  - W_t^2 \|_{\infty}) \\
%& \bigl |\mathbb{U}(A_t^{1}; Z_t^{1}, W_t^{1}) \bigr | \leq C(1 + \|A_t^{1}\|_{\infty})
%\\
%& \bigl |\mathbb{U}(A_t^{1}; Z_t^{1}, W_t^{1})	- \mathbb{U}(A_t^{2}; Z_t^{2}, W_t^{2}) \bigr |  \\
%& \leq C ( \|A_t^1  - A_t^2 \|_{\infty} + \|Z_t^1  - Z_t^2 \| _{\infty} + \|W_t^1  - W_t^2 \|_{\infty}).
\end{align*}
\end{lemma}

\begin{remark}\label{Remark_4_3_1_1_1_1}
Lemma \ref{Lemma_2} implies that the (lower and upper) value functionals are continuous with respect to $\tilde{d}$, where $\tilde{d}$ is the metric induced by $\|\cdot\|_{\infty}$. Since $d^{\circ} \leq \tilde{d}$ (see Section \ref{Section_1_1}), in view of Definition \ref{Definition_0_1}, we can easily see that the (lower and upper) value functionals are continuous with respect to $d^{\circ}$.
\end{remark}

Before stating the dynamic programming principle of the lower and upper value functionals, we introduce the \emph{backward semigroup} associated with the BSDE in (\ref{eq_2_2}). For any $s \in [t,t+\tau]$ with $\tau \in [t,T-t)$ and $b \in \mathcal{L}^2(\Omega,\mathcal{F}_{t+\tau},\mathbb{R})$, we define
\begin{align}
\label{eq_3_1}
\Pi_{s}^{t,t+\tau,A_t;U,V}[b] := \breve{y}_{s}^{t,t+\tau,A_t;U,V},~ s \in [t,t+\tau],	
\end{align}
where $\breve{y}$ is the first component of the pair $(\breve{y}_{s}^{t,t+\tau,A_t;U,V},\breve{q}_{s}^{t,t+\tau,A_t;U,V})$ that is the solution of the following BSDE:
\begin{align*}	
\breve{y}_s^{t,t+\tau,A_t;U,V} &= b + \int_s^{t+\tau} l(X_r^{t,A_t;U,V}, \breve{y}_r^{t,t+\tau, A_t;U,V}, \breve{q}_r^{t,t+\tau, A_t;U,V}, U_r,V_r) \dd r \\
&~~~ - \int_s^{t+\tau}  \breve{q}_r^{t,t+\tau, A_t;U,V} \dd B_r,~ s \in [t,t+\tau].
\end{align*}
Note that (\ref{eq_3_1}) can be regarded as a \emph{truncated} BSDE in terms of the terminal time $t+\tau$ and the terminal condition. The superscripts $t$ and $t+\tau$ indicate the initial and terminal times, respectively. By definition, we have
\begin{align}
\label{eq_3_2}
& J(t,A_t; Z_{t} \otimes u, W_{t} \otimes v) \\
%& = y_t^{t,A_t; Z_{t}^1 \otimes u, Z_{t}^2 \otimes v} \\
%& = \Pi_{t}^{t,T,A_t;Z_{t} \otimes u, W_t \otimes v} \Bigl [m(X_T^{t,A_t; Z_t \otimes u, W_{t} \otimes v}) \Bigr ]	\nonumber \\
& = \Pi_{t}^{t,t+\tau,A_t; Z_t \otimes u, W_{t} \otimes v} \Bigl [y_{t+\tau}^{t,A_t; Z_{t} \otimes u, W_{t} \otimes v} \Bigr ] \nonumber \\
%& = \Pi_{t}^{t,t+\tau,A_t; Z_{t} \otimes u, W_{t} \otimes v} \Bigl [y_{t+\tau}^{t,X_{t+\tau}^{t,A_t; Z_{t} \otimes u, W_{t} \otimes v}; Z_{t} \otimes u, W_{t} \otimes v} \Bigr ] \nonumber \\
& = \Pi_{t}^{t,t+\tau,A_t;Z_{t} \otimes u, W_{t} \otimes v} \Bigl [ J(t+\tau,X_{t+\tau}^{t,A_t; Z_{t} \otimes u, W_{t} \otimes v}; (Z_{t} \otimes u)_{t+\tau}, (W_{t} \otimes v)_{t+\tau}) \Bigr ].  \nonumber
\end{align}

Now, we state the dynamic programming principle of the lower and upper value functionals in (\ref{eq_2_6}) and (\ref{eq_2_7}).
\begin{theorem}\label{Theorem_1}
	Suppose that Assumption \ref{Assumption_1} holds. Then for any $t,t+\tau \in [0,T]$ with $t < t+\tau$, and for any $A_t \in \Lambda_t$ and $(Z_t, W_t) \in \hat{\Lambda}_t$, the lower and upper value functionals in (\ref{eq_2_6}) and (\ref{eq_2_7}), respectively, satisfy the following dynamic programming principles:
\begin{align}
\label{eq_3_3}
\mathbb{L}(A_t; Z_t, W_t) &= \essinf_{\alpha \in \mathcal{A}[t,t+\tau]} \esssup_{v \in \mathcal{V}[t,t+\tau]} \Pi_{t}^{t,t+\tau,A_t;Z_t \otimes \alpha(v), W_{t} \otimes v}  \\
&~~~~~ \Bigl [ \mathbb{L}(X_{t+\tau}^{t,A_t; Z_{t} \otimes \alpha(v), W_t \otimes v}; (Z_t \otimes \alpha(v))_{t+\tau}, (W_t \otimes v)_{t+\tau} ) \Bigr ] \nonumber \\
\label{eq_3_4}
\mathbb{U}(A_t; Z_t, W_t) &= \esssup_{\beta \in \mathcal{B}[t,t+\tau]} \essinf_{u \in \mathcal{U}[t,t+\tau]}  \Pi_{t}^{t,t+\tau,A_t; Z_{t} \otimes u, W_{t} \otimes \beta(u)}  \\
&~~~~~ \Bigl [ \mathbb{U}(X_{t+\tau}^{t,A_t; Z_{t} \otimes u, W_{t} \otimes \beta(u)}; (Z_t \otimes u)_{t+\tau}, (W_t \otimes \beta(u))_{t+\tau} ) \Bigr ]. \nonumber
\end{align}
\end{theorem}
\begin{proof}
We prove (\ref{eq_3_3}) only, as the proof for (\ref{eq_3_4}) is similar to that for (\ref{eq_3_3}). 

Let us define
\begin{align*}
\mathbb{L}^\prime (A_t;Z_t,W_t)	 & := \essinf_{\alpha \in \mathcal{A}[t,t+\tau]} \esssup_{v \in \mathcal{V}[t,t+\tau]} \Pi_{t}^{t,t+\tau,A_t;Z_{t} \otimes \alpha(v), W_{t} \otimes v}  \\
&~~~~~ \Bigl [ \mathbb{L}(X_{t+\tau}^{t,A_t; Z_{t} \otimes \alpha(v), W_t \otimes v}; (Z_t \otimes \alpha(v))_{t+\tau}, (W_t \otimes v)_{t+\tau} ) \Bigr ].
\end{align*}
Below, we show $\mathbb{L}(A_t;Z_t, W_t) \geq \mathbb{L}^\prime (A_t;Z_t, W_t)$ and $\mathbb{L}(A_t;Z_t, W_t) \leq \mathbb{L}^\prime (A_t; Z_t, W_t)$. We modify the proof of \cite{Zhang_SICON_2017} to the state and control path-dependent case.

\textit{Part (i)}: $\mathbb{L}(A_t; Z_t, W_t) \leq \mathbb{L}^\prime (A_t; Z_t, W_t)$

We first show that given $A_t \in \Lambda_t$ and $(Z_t, W_t) \in \hat{\Lambda}_t$, for any $\epsilon > 0$, there exists $\alpha^\epsilon \in \mathcal{A}[t,T]$ such that
\begin{align}
\label{eq_3_5}
	\mathbb{L}(A_t;Z_t, W_t)  \geq   \esssup_{v \in \mathcal{V}[t,T]} 	J(t,A_t; Z_{t} \otimes \alpha^\epsilon (v), W_{t} \otimes v) - \epsilon.
\end{align}

In view of \cite[Theorem A.3]{Karatzas_book_1998},
% (see also \cite{Schweizer_Note}), 
 there exists $\{\alpha_k\}$ with $\alpha_k \in \mathcal{A}[t,T]$, such that 
\begin{align}	
\label{eq_3_6}
\lim_{k \rightarrow \infty} \esssup_{v \in  \mathcal{V}[t,T]} J(t,A_t; Z_{t} \otimes \alpha_k (v), W_{t} \otimes v) \searrow \mathbb{L}(A_t;Z_t, W_t).
\end{align}
 Let $\bar{\Upsilon}_k := \{\mathbb{L}(A_t; Z_t, W_t) \geq J(t,A_t; Z_{t} \otimes \alpha_k (v), W_{t} \otimes v) - \epsilon \}$, $k \geq 1$. To make the disjoint partition of $\Omega$ with $\{\bar{\Upsilon}_k\}$, let $\Upsilon_1 : = \bar{\Upsilon}_1$ and $\Upsilon_k := \bar{\Upsilon}_k \setminus \{\cup_{i=1}^{k-1} \bar{\Upsilon}_i\}$ for $i \geq 2$. Let $\alpha^\epsilon := \sum_{k=1}^\infty \mathds{1}_{A_k} \alpha_k \in \mathcal{A}[t,T]$. Then, in view of the uniqueness of the solution to the BSDE and (\ref{eq_3_6}), we have 
\begin{align}
\label{eq_3_7}
	\mathbb{L}(A_t;Z_t, W_t) & = \sum_{k=1}^\infty \mathds{1}_{\Upsilon_k} \esssup_{v \in \mathcal{V}[t,T]}J(t,A_t; Z_{t} \otimes \alpha_k (v), W_{t} \otimes v) \\
	   & \geq \sum_{k=1}^\infty \mathds{1}_{\Upsilon_k} (J(t,A_t; Z_{t} \otimes \alpha^\epsilon (v), W_{t} \otimes v) - \epsilon) \nonumber \\
	   & = J(t,A_t; Z_{t} \otimes \alpha^\epsilon (v), W_{t} \otimes v) - \epsilon, \nonumber 
\end{align}
which shows (\ref{eq_3_5}). In fact, to show the first equality in (\ref{eq_3_7}), for any $k \geq 1$, let $\bar{\alpha} : = \mathds{1}_{\Upsilon_k} \alpha_{k} + \mathds{1}_{\Upsilon_k^C} \alpha_k^C$, where $\alpha_k, \alpha_k^C \in \mathcal{A}[t,T]$, in which $\alpha_k$  and $\alpha_k^C$ correspond to $\Upsilon_k$ and $\Upsilon_k^C$, respectively. Based on this construction, we have
\begin{align}
\label{eq_3_8}
& \esssup_{v \in \mathcal{V}[t,T]} J(t,A_t; Z_{t} \otimes \bar{\alpha} (v), W_{t} \otimes v) \nonumber \\
%& = \esssup_{v \in \mathcal{V}[t,T]} \{ \mathds{1}_{\Upsilon_k} J(t,A_t; Z_{t} \otimes \alpha_k (v), W_{t} \otimes v) + \mathds{1}_{\Upsilon_k^C} J(t,A_t; Z_{t} \otimes \alpha_k^C (v), W_{t} \otimes v) \} \nonumber \\
& \leq \mathds{1}_{\Upsilon_k} \esssup_{v \in \mathcal{V}[t,T]} J(t,A_t; Z_{t} \otimes \alpha_k (v), W_{t} \otimes v) \\
&~~~ +  \mathds{1}_{\Upsilon_k^C} \esssup_{v \in \mathcal{V}[t,T]} J(t,A_t; Z_{t} \otimes \alpha_k^C (v), W_{t} \otimes v). \nonumber
\end{align}
On the other hand, from \cite[Theorem A.3]{Karatzas_book_1998}, there exist $\{v_l\}$ and $\{v_l^C\}$ with $v_l,v_l^C \in \mathcal{V}[t,T]$ such that
\begin{align*}
	\limsup_{l \rightarrow \infty} J(t,A_t; Z_{t}^1 \otimes \alpha_k (v_l), Z_{t}^2 \otimes v) &= \esssup_{v \in \mathcal{V}[t,T]} J(t,A_t; Z_{t}^1 \otimes \alpha_k (v), Z_{t}^2 \otimes v) \\
	\limsup_{l \rightarrow \infty} J(t,A_t; Z_{t}^1 \otimes \alpha_k (v_l^C), Z_{t}^2 \otimes v) &= \esssup_{v \in \mathcal{V}[t,T]} J(t,A_t; Z_{t}^1 \otimes \alpha_k^C (v), Z_{t}^2 \otimes v).
\end{align*}
Hence, we have
\begin{align}
\label{eq_3_9}
& \esssup_{v \in \mathcal{V}[t,T]} J(t,A_t; Z_{t} \otimes \bar{\alpha} (v), W_{t} \otimes v) \nonumber \\
& \geq 	\limsup_{l \rightarrow \infty} \{\mathds{1}_{\Upsilon_k} J(t,A_t; Z_{t} \otimes \alpha_k (v_l), W_{t} \otimes v) \nonumber \\
& ~~~~~~~~~~ + \mathds{1}_{\Upsilon_k^C}  J(t,A_t; Z_{t} \otimes \alpha_k (v_l^C), W_t \otimes v) \} \nonumber \\
& = \mathds{1}_{\Upsilon_k} \esssup_{v \in \mathcal{V}[t,T]} J(t,A_t; Z_{t} \otimes \alpha_k (v), W_{t} \otimes v) \\
&~~~ +  \mathds{1}_{\Upsilon_k^C} \esssup_{v \in \mathcal{V}[t,T]} J(t,A_t; Z_{t} \otimes \alpha_k^C (v), W_{t} \otimes v). \nonumber 
\end{align}
Then, (\ref{eq_3_8}) and (\ref{eq_3_9}) imply (\ref{eq_3_7}); hence, (\ref{eq_3_5}) holds.

Now, given $A_t \in \Lambda_t$, let $\widehat{\Psi}_{t,t+\tau}^{A_t} := \bigl \{	\bar{A}_{t+\tau} \in \Lambda_{t+\tau} ~:~ \bar{a}_{r} = a_r,~ \forall r \in [0,t] \bigr \}$.
%\begin{align*}
%\widehat{\Psi}_{t,t+\tau}^{A_t} & := \bigl \{	\bar{A}_{t+\tau} \in \Lambda_{t+\tau} ~:~ \bar{a}_{r} = a_t,~ \forall r \in [t,t+\tau] \bigr \}.
%\end{align*}
We note that $\widehat{\Psi}$ is the set of continuous functions, which together with the metric $\tilde{d}$ induced by the norm $\|\cdot\|_{\infty}$ implies that $\widehat{\Psi}$ is a complete separable metric space. Let $d^\circ$ be the Skorohod metric for $\hat{\Lambda}_t$ (see the notation in Section \ref{Section_1_1}). Then in view of \cite[Theorem 12.2]{Billingsley_book}, $\hat{\Lambda}_t$ is a complete separable metric space, and from \cite{Ross_AMM_1964}, $\Psi_{t,t+\tau}^{A_t;Z_t, W_t}:= \widehat{\Psi}_{t,t+\tau}^{A_t} \times \hat{\Lambda}_t$ is a complete separable metric space with the metric $\hat{d} := \tilde{d} + d^\circ$. Therefore, there exists a countable dense subset, denoted by $\{\bar{\Psi}_k\}$ \cite{Taylor_book_1968}, and for any $(\bar{A}_{t+\tau}, Z_{t}, W_{t}) \in \Psi_{t,t+\tau}^{A_t;Z_t, W_t} $ and $\epsilon > 0$, there exists $(\bar{A}_{t+\tau}^k, \bar{Z}_{t}^k, \bar{W}_{t}^k) \in \bar{\Psi}_k$, $k \geq 1$, such that we have $\hat{d}((\bar{A}_{t+\tau}, Z_{t}, W_{t}), (\bar{A}_{t+\tau}^k, \bar{Z}_{t}^{k}, \bar{W}_{t}^{k})) < \epsilon$. For $(\bar{A}_{t+\tau}^k, \bar{Z}_{t}^{k}, \bar{W}_{t}^{k}) \in \bar{\Psi}_k$, we define the set of neighborhood of $\bar{\Psi}_k$ by
\begin{align*}
\Psi_k^\prime & := \bigl \{ 	(\bar{A}_{t+\tau}, Z_{t}, W_{t}) \in \Psi_{t,t+\tau}^{A_t;Z_t,W_t}~:~  \hat{d}((\bar{A}_{t+\tau}, Z_{t}, W_{t}), (\bar{A}_{t+\tau}^k, \bar{Z}_{t}^{k}, \bar{W}_{t}^{k})) < \epsilon \bigr \}.
\end{align*}
In view of this construction, $\cup_{k=1}^{\infty} \Psi_k^\prime = \Psi_{t,t+\tau}^{A_t; Z_t, W_t}$, and by a slight abuse of notation, with $\Psi_1^\prime := \Psi_1^\prime$ and $\Psi_k^\prime := \Psi_k^\prime \setminus \{\cup_{j=1}^{k-1} \Psi_k^\prime\}$, $k \geq 2$, we still have $\cup_{k=1}^{\infty} \Psi_k^\prime = \Psi_{t,t+\tau}^{A_t;Z_t, W_t}$, where $\{\Psi_k^\prime\}$ is the disjoint partition of $\Psi_{t,t+\tau}^{A_t; Z_t, W_t}$. 

For any $(A_t,Z_t, W_t) \in \Lambda_t \times \hat{\Lambda}_t$, $\alpha^\prime \in \mathcal{A}[t,t+\tau]$ and $ v^\prime \in \mathcal{V}[t,t+\tau]$, with the above construction, together with Lemma \ref{Lemma_2} and Remark \ref{Remark_4_3_1_1_1_1}, for each $\epsilon > 0$, there exists a constant $C > 0$ such that
\begin{align}	
\label{eq_3_11}
& \mathbb{L}(X_{t+\tau}^{t,A_t;Z_{t} \otimes \alpha^\prime(v^\prime), W_t \otimes v^\prime}; (Z_t \otimes \alpha^\prime(v^\prime))_{t+\tau}, (W_t \otimes v^\prime)_{t+\tau}) \\
& = \sum_{k=1}^\infty \mathds{1}_{(X_{t+\tau}^{t,A_t;Z_{t} \otimes \alpha^\prime(v^\prime), W_t \otimes v^\prime}, Z_{t}, W_{t}) \in \Psi_k^\prime  } \nonumber \\
&~~~~~~~~~~ \times \mathbb{L}(X_{t+\tau}^{t,A_t;Z_{t} \otimes \alpha^\prime(v^\prime), W_t \otimes v^\prime}; (Z_t \otimes \alpha^\prime(v^\prime))_{t+\tau}, (W_t \otimes v^\prime)_{t+\tau}) \nonumber \\
& \geq \sum_{k=1}^\infty  \mathds{1}_{(X_{t+\tau}^{t,A_t; Z_{t} \otimes \alpha^\prime(v^\prime), W_t \otimes v^\prime}, Z_{t}, W_{t}) \in \Psi_k^\prime  } \nonumber \\
&~~~~~~~~~~ \times \mathbb{L}(\bar{A}_{t+\tau}^k; (\bar{Z}_{t}^{k} \otimes \alpha^\prime (v^\prime))_{t+\tau}, (\bar{W}_{t}^{k} \otimes v^\prime)_{t+\tau}) - C \epsilon. \nonumber
\end{align}
Note that (\ref{eq_3_5}) implies that there exists $\alpha_k^{\prime \prime} \in \mathcal{A}[t+\tau,T]$ such that for $k \geq 1$,
\begin{align*}	
& \mathbb{L}(\bar{A}_{t+\tau}^k; (\bar{Z}_{t}^{k} \otimes \alpha^\prime (v^\prime))_{t+\tau}, (\bar{W}_{t}^{k} \otimes v^\prime)_{t+\tau}) \\
& \geq \sup_{v \in \mathcal{V}[t+\tau,T]} J(t+\tau,\bar{A}_{t+\tau}^k; (\bar{Z}_{t}^{k} \otimes \alpha_k^{\prime \prime} (v))_{t+\tau}, (\bar{W}_{t}^{k} \otimes v )_{t+\tau}) - \epsilon.
\end{align*}
Hence, from (\ref{eq_3_11}), for any $v^{\prime \prime} \in \mathcal{V}[t+\tau,T]$, we have
\begin{align}
\label{eq_3_11_1}
& \mathbb{L}(\bar{A}_{t+\tau}^k; (\bar{Z}_{t}^{k} \otimes \alpha^\prime (v^\prime))_{t+\tau}, (\bar{W}_{t}^{k} \otimes v^\prime)_{t+\tau}) \\
& \geq \sum_{k=1}^\infty  \mathds{1}_{(X_{t+\tau}^{t,A_t; Z_{t} \otimes \alpha^\prime(v^\prime), W_t \otimes v^\prime}, Z_{t}, W_{t}) \in \Psi_k^\prime  } \nonumber \\
&~~~~~~~~~~ \times J(t+\tau,\bar{A}_{t+\tau}^k; (\bar{Z}_{t}^{k} \otimes \alpha_k^{\prime \prime} (v^{\prime \prime}))_{t+\tau}, (\bar{W}_{t}^{k} \otimes v^{\prime \prime} )_{t+\tau}) - C \epsilon \nonumber \\
& = J(t+\tau,X_{t+\tau}^{t,A_t;Z_{t} \otimes \alpha^\prime(v^\prime),W_t \otimes v^\prime}; (Z_{t} \otimes \alpha^{\prime \prime} (v^{\prime \prime}))_{t+\tau}, (W_{t} \otimes v^{\prime \prime} )_{t+\tau}) - C \epsilon, \nonumber
\end{align}
where $\alpha^{\prime \prime} (v^{\prime \prime}) := \sum_{k=1}^\infty \mathds{1}_{(X_{t+\tau}^{t,A_t;Z_{t} \otimes \alpha^\prime(v^\prime), W_t \otimes v^\prime}, Z_{t}, W_{t}) \in \Psi_k^\prime  } \alpha_k^{\prime \prime} (v^{\prime \prime}) $ and $v^{\prime \prime} \in \mathcal{V}[t+\tau,T]$.

Let us define
\begin{align*}
	v^{\prime \prime \prime}_s & := \mathds{1}_{s \in [t,t+\tau]} v_s^{\prime} + \mathds{1}_{s \in (t+\tau,T]} v^{\prime \prime}_s \\
	\alpha^{\prime \prime \prime}_s &:= \mathds{1}_{s \in [t,t+\tau]} \alpha_s^\prime (v^{\prime}) + \mathds{1}_{s \in (t+\tau,T]} \alpha_s^{\prime \prime} (v^{\prime \prime}).
\end{align*}
Note that $v^{\prime \prime \prime} \in \mathcal{V}[t,T]$ and $\alpha^{\prime \prime \prime} \in \mathcal{A}[t,T]$. In view of (\ref{eq_2_4}), we have
\begin{align*}
(W_t \otimes v^{\prime \prime \prime})[s] & = \mathds{1}_{s \in [t,t+\tau]} (W_t \otimes v^{\prime})[s] + \mathds{1}_{s \in (t+\tau,T]} (W_t \otimes v^{\prime \prime})[s]  \\
(Z_t \otimes \alpha^{\prime \prime \prime} ( v^{\prime \prime \prime}))[s]  & = 	\mathds{1}_{s \in [t,t+\tau]} (Z_t \otimes \alpha^\prime (v^\prime) )[s] + \mathds{1}_{s \in (t+\tau,T]} (Z_t \otimes \alpha^{\prime \prime} (v^{\prime \prime}))[s],
\end{align*}
where it can be verified that $W_t \otimes v^{\prime \prime \prime}  \in \mathcal{V}[0,T]$.
% and $Z_t \otimes \alpha^{\prime \prime \prime} (v^{\prime \prime \prime}) \in \mathcal{A}[0,T]$. 
Then from the comparison principle in (iv) of Lemma \ref{Lemma_1}, (\ref{eq_3_1}) and (\ref{eq_3_2}), we have
\begin{align*}
& \Pi_{t}^{t,t+\tau,A_t;Z_{t} \otimes \alpha^\prime(v^\prime), W_{t} \otimes v^\prime}  \\
&~~ \Bigl [ \mathbb{L}(X_{t+\tau}^{t,A_t;Z_{t} \otimes \alpha^\prime(v^\prime),W_t \otimes v^\prime}; (Z_t \otimes \alpha^\prime(v^\prime))_{t+\tau}, (W_t \otimes v^\prime)_{t+\tau}) \Bigr ]  \\
& \geq \Pi_{t}^{t,t+\tau,A_t;Z_{t} \otimes \alpha^\prime(v^\prime), W_{t} \otimes v^\prime}  \\
&~~~~~ \Bigl [J(t+\tau,X_{t+\tau}^{t,A_t;Z_{t} \otimes \alpha^\prime(v^\prime), W_t \otimes v^\prime}; (Z_{t} \otimes \alpha^{\prime \prime} (v^{\prime \prime}))_{t+\tau}, (W_{t} \otimes v^{\prime \prime} )_{t+\tau}) \Bigr ] - C \epsilon \\
& = J(t,A_t; (Z_{t} \otimes \alpha^{\prime \prime \prime} (v^{\prime \prime \prime}))_{t+\tau}, (W_{t} \otimes v^{\prime \prime \prime} )_{t+\tau}) - C \epsilon.
\end{align*}
The arbitrariness of $v^{\prime \prime \prime}$ and $\alpha^{\prime \prime \prime}$, together with the definition of $\Pi$ and (\ref{eq_3_2}), yields
\begin{align*}
& \essinf_{\alpha \in [t,t+\tau]} \esssup_{v \in \mathcal{V}[t,t+\tau]}	\Pi_{t}^{t,t+\tau,A_t; Z_{t} \otimes \alpha (v), W_{t} \otimes v} \\
&~~~~~ \Bigl [ \mathbb{L}(X_{t+\tau}^{t,A_t;Z_{t} \otimes \alpha(v), W_t \otimes v}; (Z_t \otimes \alpha(v))_{t+\tau}, (W_t \otimes v)_{t+\tau}) \Bigr ]  \\
& \geq \essinf_{\alpha \in [t,T]} \esssup_{v \in \mathcal{V}[t,T]} J(t,A_t; (Z_{t} \otimes \alpha (v))_{t+\tau}, (W_{t} \otimes v )_{t+\tau}) - C \epsilon.
\end{align*}
By letting $\epsilon \downarrow 0$, we have the desired result. 

\textit{Part (ii)}: $\mathbb{L}(A_t;Z_t,W_t) \geq \mathbb{L}^\prime (A_t;Z_t,W_t)$

We first note that for any fixed $\alpha^\prime \in \mathcal{A}[t,T]$ with $v \in \mathcal{V}[t,T]$, its restriction to $[t,t+\tau]$ is still nonanticipative  independent of any special choice of $v \in \mathcal{V}[t+\tau, T]$, i.e., $\alpha^\prime|_{[t,t+\tau]} \in \mathcal{A}[t,t+\tau]$ for $v \in \mathcal{V}[t,t+\tau]$, due to the nonanticipative property of $\alpha^\prime$. Recall the definition of $\mathbb{L}^\prime$; then with the restriction of $\alpha^\prime$ to $[t,t+\tau]$, we have
\begin{align*}
\mathbb{L}^\prime (A_t;Z_t,W_t)	 & \leq \esssup_{v \in \mathcal{V}[t,t+\tau]} \Pi_{t}^{t,t+\tau,A_t;Z_{t} \otimes \alpha^\prime(v), W_{t} \otimes v}  \\
&~~~~~ \Bigl [ \mathbb{L}(X_{t+\tau}^{t,A_t;Z_{t} \otimes \alpha^\prime(v), W_t \otimes v}; (Z_t \otimes \alpha^\prime(v))_{t+\tau}, (W_t \otimes v)_{t+\tau} ) \Bigr ].	
\end{align*}
Furthermore, similar to (\ref{eq_3_6}), there exists $\{v_k\}$, with $v_k \in \mathcal{V}[t,t+\tau]$, such that
\begin{align}
\label{eq_3_12}
& \lim_{k \rightarrow \infty} \Pi_{t}^{t,t+\tau,A_t;Z_{t} \otimes \alpha^\prime(v_k), W_{t} \otimes v_k}  \\
&~~~~~~~~ \Bigl [ \mathbb{L}(X_{t+\tau}^{t,A_t; Z_{t} \otimes \alpha^\prime(v_k),W_t \otimes v_k}; (Z_t \otimes \alpha^\prime(v_k))_{t+\tau}, (W_t \otimes v_k)_{t+\tau} ) \Bigr ] \nonumber \\
&=  \esssup_{v \in \mathcal{V}[t,t+\tau]} \Pi_{t}^{t,t+\tau,A_t; Z_{t} \otimes \alpha^\prime(v), W_{t} \otimes v}  \nonumber \\
&~~~~~~~~ \Bigl [ \mathbb{L}(X_{t+\tau}^{t,A_t; Z_{t} \otimes \alpha^\prime(v), W_t \otimes v}; (Z_t \otimes \alpha^\prime(v))_{t+\tau}, (W_t \otimes v)_{t+\tau} ) \Bigr ].	 \nonumber
\end{align}
Then by using (\ref{eq_3_12}) and the approach of (\ref{eq_3_5}), for each $\epsilon \geq 0$, there exists $v ^\prime \in \mathcal{V}[t,t+\tau]$ such that
\begin{align}
\label{eq_4_13_1_1}
\mathbb{L}^\prime (A_t;Z_t,W_t)	 & \leq   \Pi_{t}^{t,t+\tau,A_t;Z_{t} \otimes \alpha^\prime(v^\prime), W_{t} \otimes v^\prime}   \\
&  \Bigl [\mathbb{L}(X_{t+\tau}^{t,A_t; Z_{t} \otimes \alpha^\prime(v^\prime), W_t \otimes v^\prime}; (Z_t \otimes \alpha^\prime(v ^\prime))_{t+\tau}, (W_t \otimes v ^\prime)_{t+\tau} ) \Bigr ]  + \epsilon. 	 \nonumber
\end{align}
Similar to the argument and the notation introduced in (\ref{eq_3_11}) and (\ref{eq_3_11_1}), there exists $v_k^{\prime \prime}   \in \mathcal{V}[t+\tau,T]$, $k \geq 1$, such that 
\begin{align*}	
& \mathbb{L}(X_{t+\tau}^{t,A_t; Z_{t} \otimes \alpha^\prime(v^\prime), W_t \otimes v^\prime}; (Z_t \otimes \alpha^\prime (v ^\prime))_{t+\tau}, (W_t \otimes v ^\prime)_{t+\tau} ) \\
& = \sum_{k=1}^\infty \mathds{1}_{(X_{t+\tau}^{t,A_t; Z_{t} \otimes \alpha^\prime(v^\prime), W_t \otimes v^\prime}, Z_{t}, W_{t}) \in \Psi_k^\prime  } \\
&~~~~~~~~~~ \times \mathbb{L}(X_{t+\tau}^{t,A_t; Z_{t} \otimes \alpha^\prime(v^\prime), W_t \otimes v^\prime}; (Z_t \otimes \alpha^\prime(v^\prime))_{t+\tau}, (W_t \otimes v^\prime)_{t+\tau}) \\
& \leq \sum_{k=1}^\infty  \mathds{1}_{(X_{t+\tau}^{t,A_t; Z_{t} \otimes \alpha^\prime(v^\prime), W_t \otimes v^\prime}, Z_{t}, W_{t}) \in \Psi_k^\prime  } \\
&~~~~~~~~~~ \times J(t+\tau,\bar{A}_{t+\tau}^k; (\bar{Z}_{t}^k \otimes \alpha^{\prime } (v^{\prime \prime}_k))_{t+\tau}, (\bar{W}_{t}^{k} \otimes v^{\prime \prime}_k )_{t+\tau}) + C \epsilon \\
& = J(t+\tau,X_{t+\tau}^{t,A_t; Z_{t} \otimes \alpha^\prime(v^\prime), W_t \otimes v^\prime}; (Z_{t} \otimes \alpha^{\prime} |_{[t+\tau,T]}(v^{\prime \prime}))_{t+\tau}, (W_{t} \otimes v^{\prime \prime} )_{t+\tau}) + C \epsilon,
\end{align*}
where $v^{\prime \prime}:= \sum_{k=1}^\infty \mathds{1}_{(X_{t+\tau}^{t,A_t; Z_{t} \otimes \alpha^\prime(v^\prime), W_t \otimes v^\prime}, Z_{t}, W_{t}) \in \Psi_k^\prime  } v^{\prime \prime}_k$.

Let us define
\begin{align*}
v^{\prime \prime \prime}_s & := \mathds{1}_{s \in [t,t+\tau]} v_s^{\prime} + \mathds{1}_{s \in (t+\tau,T]} v_s^{\prime \prime} \\
\alpha^{\prime}_s ( v^{\prime \prime \prime}) & := 	\mathds{1}_{s \in [t,t+\tau]} \alpha^\prime_s (v^\prime) + \mathds{1}_{s \in (t+\tau,T]} \alpha^{\prime }_s (v^{\prime \prime}). 
\end{align*}
Note that $v^{\prime \prime \prime} \in \mathcal{V}[t,T]$ and $\alpha^\prime \in \mathcal{A}[t,T]$. Then, from (\ref{eq_2_4}),
\begin{align*}	
(W_t \otimes v^{\prime \prime \prime})[s] & = \mathds{1}_{s \in [t,t+\tau]} (W_t \otimes v^{\prime})[s] + \mathds{1}_{s \in (t+\tau,T]} (W_t \otimes v^{\prime \prime})[s] \\
(Z_t \otimes \alpha^{\prime} ( v^{\prime \prime \prime}))[s] & = 	\mathds{1}_{s \in [t,t+\tau]} (Z_t \otimes  \alpha^\prime (v^\prime))[s] + \mathds{1}_{s \in (t+\tau,T]} (Z_t \otimes  \alpha^{\prime } (v^{\prime \prime}))[s],
\end{align*}
where $W_t \otimes v^{\prime \prime \prime} \in \mathcal{V}[0,T]$. From (iv) of Lemma \ref{Lemma_1}, (\ref{eq_3_1}), and (\ref{eq_3_2}), we have
\begin{align*}
& \mathbb{L}^\prime (A_t;Z_t,W_t)	 \\
& \leq   \Pi_{t}^{t,t+\tau,A_t;Z_{t} \otimes \alpha^\prime(v^\prime), W_{t} \otimes v^\prime}  \\
&~~~~~ \Bigl [ J(t+\tau,X_{t+\tau}^{t,A_t;Z_{t} \otimes \alpha^\prime(v^\prime), W_t \otimes v^\prime}; (Z_{t} \otimes \alpha^{\prime} (v^{\prime \prime}))_{t+\tau}, (W_{t} \otimes v^{\prime \prime} )_{t+\tau}) \Bigr ] + C \epsilon \\
& = J(t,A_t; Z_{t} \otimes \alpha^{\prime } (v^{\prime \prime \prime}), W_{t} \otimes v^{\prime \prime \prime} ) + C \epsilon.
\end{align*}
The arbitrariness of $v^{\prime \prime \prime}$ and the definition of $\Pi$ imply, 
\begin{align*}
& \mathbb{L}^\prime (A_t; Z_t, W_t)	\leq \sup_{v \in \mathcal{V}[t,T]}  J(t,A_t; Z_{t} \otimes \alpha^{\prime } (v), W_{t} \otimes v )  + C \epsilon,
\end{align*}
and by taking $\essinf$ with respect to $\alpha \in \mathcal{A}[t,T]$ and then $\epsilon \downarrow 0$, we have the desired result. Hence, Parts (i) and (ii) show the dynamic programming principle of the lower value functional $\mathbb{L}$ in (\ref{eq_3_3}). This completes the proof of the theorem.
\end{proof}

\begin{remark}
By using standard localization and approximation techniques, Theorem \ref{Theorem_1} can be extended to stopping times.	
\end{remark}

From Lemma \ref{Lemma_2}, the (lower and upper) value functionals are continuous with respect to the initial state and control paths. We next state the continuity of the (lower and upper) value functionals in $t \in [0,T]$. 

\begin{lemma}\label{Lemma_3}
	Suppose that Assumption \ref{Assumption_1} holds. Then, the lower and upper value functionals are  continuous in $t$. In particular, there exists a constant $C > 0$ such that for any $(A_T,Z_T,W_T) \in \Lambda_T \times \hat{\Lambda}_T$ and $t_1,t_2 \in [0,T]$ with $t^\prime := \max\{t_1,t_2\}$, ($\mathbb{G} := \mathbb{L},\mathbb{U}$)
\begin{align*}
% |\mathbb{L}(A_{t_1}; Z_{t_1}, W_{t_1}) - \mathbb{L}(A_{t_2}; Z_{t_2}, W_{t_2})| &\leq C (1 + \|A_{t^\prime} \|_{\infty} )	|t_1 - t_2|^{1/2} \\
|\mathbb{G}(A_{t_1}; Z_{t_1}, W_{t_1}) - \mathbb{G}(A_{t_2}; Z_{t_2}, W_{t_2})| &\leq C (1 + \|A_{t^\prime} \|_{\infty})	|t_1 - t_2|^{1/2}.
\end{align*}
\end{lemma}
\begin{proof}
	We prove the case for the lower value functional only, since the proof for the upper value functional is similar.  Without loss of generality, for any $t_{1} = t, t_{2} = t+\tau \in [0,T]$ with $t < t+\tau$, we need to prove
	\begin{align}	
	\label{eq_3_13}
	- C (1 + \|A_{t+\tau} \|_{\infty})	\tau^{1/2} & \leq \mathbb{L}(A_t; Z_t, W_t)  - \mathbb{L}(A_{t+\tau}; Z_{t+\tau}, W_{t+\tau})  \\
	&\leq C (1 + \|A_{t+\tau} \|_{\infty})	\tau^{1/2}. \nonumber
	\end{align}
	In view of the dynamic programming principle (\ref{eq_3_3}) in Theorem \ref{Theorem_1} and (\ref{eq_3_5}), for any $\epsilon > 0$, there exists $ \alpha^\epsilon \in \mathcal{A}[t,t+\tau]$ such that for any $v \in \mathcal{V}[t,t+\tau]$,
	\begin{align*}
	\mathbb{L}(A_t; Z_t, W_t) & \geq \Pi_{t}^{t,t+\tau,A_t;Z_{t} \otimes \alpha^\epsilon(v), W_{t} \otimes v}  \\
&~~~~~  \Bigl [ \mathbb{L}(X_{t+\tau}^{t,A_t; Z_{t} \otimes \alpha^\epsilon(v), W_t \otimes v}; (Z_t \otimes \alpha^\epsilon(v))_{t+\tau}, (W_t \otimes v)_{t+\tau} ) \Bigr ] - \epsilon.
	\end{align*}
The definition of $\Pi$ implies
\begin{align}	
\label{eq_3_15_1}
 \mathbb{L}(A_t; Z_t, W_t) - \mathbb{L}(A_{t+\tau}; Z_{t+\tau}, W_{t+\tau})  \geq  L^{(1)} + L^{(2)} - L^{(3)} - \epsilon,
\end{align}
where
\begin{align*}
L^{(1)} & := 	\Pi_{t}^{t,t+\tau,A_t;Z_{t} \otimes \alpha^\epsilon(v), W_{t} \otimes v}  \\
&~~~~~ \Bigl [ \mathbb{L}(X_{t+\tau}^{t,A_t; Z_{t} \otimes \alpha^\epsilon(v), W_t \otimes v};(Z_t \otimes \alpha^\epsilon (v))_{t+\tau},  (W_t \otimes v)_{t+\tau} ) \Bigr ] \\
&~~~ - \Pi_{t}^{t,t+\tau,A_t; Z_{t} \otimes \alpha^\epsilon(v), W_{t} \otimes v}  \Bigl [\mathbb{L}(A_{t+\tau}; (Z_t \otimes \alpha^\epsilon(v))_{t+\tau}, (W_t \otimes v)_{t+\tau}) \Bigr ] 
\\
L^{(2)} & := \Pi_{t}^{t,t+\tau, A_t;Z_{t} \otimes \alpha^\epsilon (v), W_{t} \otimes v}  \Bigl [\mathbb{L}(A_{t+\tau}; (Z_t \otimes \alpha^\epsilon (v))_{t+\tau}, (W_t \otimes v)_{t+\tau}) \Bigr ] \\
&~~~ - \Pi_{t}^{t,t+\tau, A_t; Z_{t} \otimes \alpha^\epsilon (v), W_{t} \otimes v}  \Bigl [\mathbb{L}(A_{t+\tau}; Z_{t+\tau}, W_{t+\tau}  )\Bigr ]
\\
L^{(3)} & :=  \Pi_{t}^{t,t+\tau,A_t;Z_{t} \otimes \alpha^\epsilon (v), W_{t} \otimes v}  \Bigl [\mathbb{L}(A_{t+\tau}; Z_{t+\tau}, W_{t+\tau}  )\Bigr ] \\
&~~~ - \mathbb{L}(A_{t+\tau}; Z_{t+\tau}, W_{t+\tau}) .
\end{align*}
Note that $L^{(i)} \geq - |L^{(i)}|$ for $i=1,2,3$. 

Now, Lemmas \ref{Lemma_1} and \ref{Lemma_2}, the definition of $\Pi$, and Jensen's inequality imply that there exist a constant $C>0$ such that
\begin{align}
\label{eq_3_15}
|L^{(1)}| & \leq C \mathbb{E} \Bigl [ \bigl | \mathbb{L}(X_{t+\tau}^{t,A_t;Z_{t} \otimes \alpha^\epsilon (v), W_t \otimes v};(Z_t \otimes \alpha^\epsilon (v))_{t+\tau},  (W_t \otimes v)_{t+\tau} ) \\
&~~~~~~~~~~ - \mathbb{L}(A_{t+\tau}; (Z_t \otimes \alpha^\epsilon (v))_{t+\tau}, (W_t \otimes v)_{t+\tau}) \bigr |^2  \bigl | \mathcal{F}_t \Bigr ]^{1/2} \nonumber \\
& \leq C \mathbb{E} \Bigl [ \bigl \| X_{t+\tau}^{t,A_t;Z_{t} \otimes \alpha^\epsilon (v), W_t \otimes v}  - A_{t+\tau} \bigr \|_\infty^2 \bigl | \mathcal{F}_t \Bigr]^{1/2}  \leq C (1 + \|A_{t+\tau}\|_\infty)\tau^{1/2}. \nonumber
\end{align}
Moreover, from the definition of $\Pi$ and Proposition \ref{Proposition_1}, $L^{(3)}$ is equivalent to
\begin{align*}
	L^{(3)} 
& = \mathbb{E} \Bigl [ \int_t^{t+\tau} l(X_r^{t,A_t;Z_t \otimes \alpha^\epsilon(v), W_t \otimes v}, y_r^{t,t+\tau,A_t; Z_t \otimes \alpha^\epsilon(v), W_t \otimes v}, \\
&~~~~~~~~~~ q_r^{t,t+\tau,A_t; Z_t \otimes \alpha^\epsilon(v), W_t \otimes v}, (Z_t \otimes \alpha^\epsilon (v))_{r}, (W_t \otimes v)_{r}) \dd r \bigl | \mathcal{F}_t \Bigr ].
\end{align*}
Then H\"older inequality, Assumption \ref{Assumption_1}, Lemma \ref{Lemma_1} imply that
\begin{align}
\label{eq_3_16}
	|L^{(3)}| 
%& \leq \tau^{1/2} \mathbb{E} \Bigl [ \int_t^{t+\tau} | l(X_r^{t,A_t;Z_t \otimes \alpha^\epsilon (v), W_t \otimes v}, y_r^{t,t+\tau,A_t; Z_t \otimes \alpha^\epsilon (v),W_t \otimes v}, \nonumber \\
%&~~~~~~~~~~ q_r^{t,t+\tau,A_t; Z_t \otimes \alpha^\epsilon (v), W_t \otimes v}, (Z_t \otimes \alpha^\epsilon (v))_{r}  , (W_t \otimes v)_{r}) |^2 \dd r \bigl | \mathcal{F}_t \Bigr ]^{1/2} \nonumber \\
& \leq C \tau^{1/2} \mathbb{E} \Bigl [ \int_t^{t+\tau} \bigl [ 1 + \|X_r^{t,A_t; Z_t \otimes \alpha^\epsilon(v), W_t \otimes v}\|^2_{\infty} + |y_r^{t,t+\tau,A_t; Z_t \otimes \alpha^\epsilon(v), W_t \otimes v}|^2 \nonumber \\
&~~~~~~~~~~ + |q_r^{t,t+\tau,A_t;Z_t \otimes \alpha^\epsilon(v), W_t \otimes v}|^2 \bigr ]\dd r \bigl | \mathcal{F}_t \Bigr ]^{1/2}  \leq C ( 1 + \|A_{t+\tau}\|_\infty)\tau^{1/2}.
\end{align}
Furthermore, in view of the definitions of the lower value functional in (\ref{eq_2_6}) and the objective functional in (\ref{eq_2_5}), we have
\begin{align*}	
& \mathbb{L}(A_{t+\tau}; Z_{t+\tau}, W_{t+\tau}  ) \\
&= \essinf_{\alpha  \in \mathcal{A}[t+\tau,T]} \esssup_{ v \in \mathcal{V}[t+\tau,T]} J(t+\tau, A_{t+\tau}; Z_{t+\tau} \otimes \alpha(v), W_{t+\tau} \otimes v) \\
&\mathbb{L}(A_{t+\tau}; (Z_t \otimes \alpha(v))_{t+\tau}, (W_t \otimes v)_{t+\tau}) \\
& = \essinf_{\alpha  \in \mathcal{A}[t+\tau,T]} \esssup_{ v \in \mathcal{V}[t+\tau,T]} J(t+\tau, A_{t+\tau}; Z_{t} \otimes \alpha(v), W_{t} \otimes v).
\end{align*} 
%\begin{align}
%L^{(2)} & := \Pi_{t}^{t,t+\tau,A_t;Z_{t}^1 \otimes \alpha^\epsilon (v), Z_{t}^2 \otimes v}  \Bigl [\mathbb{L}(t+\tau,A_{t+\tau}; (Z_t^{1} \otimes \alpha^\epsilon (v))_{t+\tau}, (Z_t^{2} \otimes v)_{t+\tau}) \Bigr ] \\
%&~~~ - \Pi_{t}^{t,t+\tau,A_t;Z_{t}^1 \otimes \alpha^\epsilon (v), Z_{t}^2 \otimes v}  \Bigl [\mathbb{L}(t+\tau,A_{t+\tau}; Z_{t+\tau}^{1}, Z_{t+\tau}^{2}  )\Bigr ]
%\end{align}
From (iii) of Lemmas \ref{Lemma_1}, Lemma \ref{Lemma_2}, (\ref{eq_2_4}), and the definition of $\Pi$, we have
\begin{align}
\label{eq_3_17}
|L^{(2)}|  
& \leq C \mathbb{E} \Big [ \bigl | \mathbb{L}(A_{t+\tau}; (Z_t \otimes \alpha^\epsilon(v))_{t+\tau}, (W_t \otimes v)_{t+\tau})  \\
&~~~~~~~~~~ - \mathbb{L}(A_{t+\tau}; Z_{t+\tau}, W_{t+\tau} )\bigr |^2 | \mathcal{F}_t \Big]^{1/2} \nonumber \\
& \leq C \mathbb{E} \Bigl [ \int_{t+\tau}^T \bigl [ \|(Z_{t+\tau} \otimes \alpha^\epsilon (v))_r 	- (Z_t \otimes \alpha^\epsilon (v))_r \|^2_{\infty}   \\
&~~~~~~~~~~ + \|(W_{t+\tau} \otimes v)_r 	- (W_t \otimes v)_r \|^2_{\infty}  \bigr ] \dd r \bigl | \mathcal{F}_t \Bigr ]^{1/2}  = 0. \nonumber
\end{align}

By substituting (\ref{eq_3_15})-(\ref{eq_3_17}) into (\ref{eq_3_15_1}), 
%$\mathbb{L}(A_t; Z_t, W_t)  - \mathbb{L}(A_{t+\tau}; Z_{t+\tau}, W_{t+\tau})\geq - C (1 + \|A_{t+\tau} \|_{\infty})	\tau^{1/2} - \epsilon$.
\begin{align*}	
\mathbb{L}(A_t; Z_t, W_t)  - \mathbb{L}(A_{t+\tau}; Z_{t+\tau}, W_{t+\tau})\geq - C (1 + \|A_{t+\tau} \|_{\infty})	\tau^{1/2} - \epsilon.
\end{align*}
Hence, the arbitrariness of $\epsilon$ implies the first inequality part in (\ref{eq_3_13}). The second inequality part in (\ref{eq_3_13}) can be shown in a similar way. We complete the proof.
\end{proof}

%Based on Lemmas \ref{Lemma_2} and \ref{Lemma_3}, the (lower and upper) value functional satisfy the following regularity condition:
%\begin{proposition}\label{Proposition_2}
%	Suppose that Assumption \ref{Assumption_1} holds. Then there exists a constant $C > 0$ such that for $t_1,t_2 \in [t,T]$ with $t_1 \leq t_2$ and any $A_{t_1}^1,A_{t_2}^2 \in \Lambda$ and $(Z_{t_i}, W_{t_i}) \in \hat{\Lambda}$, $i=1,2$, we have
%	\begin{align*}
%	& |\mathbb{L}(A_{t_1}^1;Z_{t_1}, W_{t_1}) - 	\mathbb{L}(A_{t_2}^2; Z_{t_2}, W_{t_2}) | \\
%	& \leq C \Bigl ( \|A_{t_2}^1 - A_{t_2}^2 \|_{\infty} + (1 + \|A_{t_2}^1\|_{\infty})(t_2 - t_1)^{1/2} \Bigr ) \\
%	& |\mathbb{U}(A_{t_1}^1; Z_{t_1}, W_{t_1}) - 	\mathbb{U}(A_{t_2}^2;Z_{t_2}, W_{t_2}) | \\
%	& \leq C \Bigl ( \|A_{t_2}^1 - A_{t_2}^2 \|_{\infty} + (1 + \|A_{t_1}^1\|_{\infty} + \|A_{t_2}^2\|_{\infty})(t_2 - t_1)^{1/2} \Bigr ).
%%	& |\mathbb{L}(A_{t_1}^1;Z_{t_1}, W_{t_1}) - 	\mathbb{L}(A_{t_2}^2; Z_{t_2}, W_{t_2}) | \\
%%	& \leq C \Bigl ( \|A_{t_1,t_2-t_1}^1 - A_{t_2}^2 \|_{\infty} + (1 + \|A_{t_1}^1\|_{\infty} + \|A_{t_2}^2\|_{\infty})(t_2 - t_1)^{1/2} \Bigr ) \\
%%	& |\mathbb{U}(A_{t_1}^1; Z_{t_1}, W_{t_1}) - 	\mathbb{U}(A_{t_2}^2;Z_{t_2}, W_{t_2}) | \\
%%	& \leq C \Bigl ( \|A_{t_1,t_2-t_1}^1 - A_{t_2}^2 \|_{\infty} + (1 + \|A_{t_1}^1\|_{\infty} + \|A_{t_2}^2\|_{\infty})(t_2 - t_1)^{1/2} \Bigr ).
%	\end{align*}
%\end{proposition}

\section{State and Control Path-Dependent Hamilton-Jacobi-Isaacs Equations and Viscosity Solutions}\label{Section_4}

In this section, we introduce the lower and upper state and control path-dependent Hamilton-Jacobi-Isaacs (PHJI) equations that are path-dependent nonlinear second-order PDEs (PPDEs). We show that the (lower and upper) value functionals are viscosity solutions of the corresponding PHJI equations.

The Hamiltonian, $\mathcal{H}:\Lambda \times \hat{\Lambda} \times \mathbb{R} \times \mathbb{R}^{n} \times \mathbb{S}^n \rightarrow \mathbb{R}$, is defined by
\begin{align}
\label{eq_5_1}
& \mathcal{H}(A_t,Z_t^{u}, W_t^{v}, y, p, P)	 \\
& =  \langle f(A_t,Z_t^{u}, W_t^{v}), p \rangle  + l(A_t, y, \langle p, \sigma(A_t,Z_t^{u}, W_t^{v}) \rangle, Z_t^{u}, W_t^{v}) \nonumber \\
&~~~ + \frac{1}{2} \Tr (P \sigma(A_t, Z_t^{u}, W_t^{v}) \sigma^\top(A_t, Z_t^{u}, W_t^{v})), \nonumber
\end{align}
where
\begin{align}
\label{eq_5_1_1_1_1_1_1_1}
Z_t^{u}[s] := \begin{cases}
 	z_s, & \text{if}~ s \in [0,t) \\
 u & \text{if}~ s = t,
 \end{cases}	~~ W_t^{v}[s] := \begin{cases}
 	w_s, & \text{if}~ s \in [0,t) \\
 v & \text{if}~ s = t.
 \end{cases}	
\end{align}
With (\ref{eq_5_1}), we introduce the lower PHJI equation
\begin{align}
\label{eq_5_2}
\begin{cases}
\mathbb{H}^-(A_t,Z_t, W_t, (\partial_t \mathbb{L}^{u,v}, \mathbb{L}, \partial_x \mathbb{L}, \partial_{xx} \mathbb{L})(A_t;Z_t, W_t) ) \\
: = \sup_{v \in \mathrm{V}}	 \inf_{u \in \mathrm{U}} \Bigl \{	\partial_t \mathbb{L}(A_t;Z_t^{u}, W_t^{v}) \\
~~~~~~~~~~+ \mathcal{H}(A_t,Z_t^u, W_t^v,(\mathbb{L}, \partial_x \mathbb{L}, \partial_{xx} \mathbb{L})(A_t;Z_t, W_t) ) \Bigr \} = 0 ,\\
~~~~~~~~~~~~ t \in [0,T),~ (A_t,Z_t, W_t) \in \Lambda \times \hat{\Lambda} \\
	\mathbb{L}(A_T;Z_T, W_T) = m(A_T),~ (A_T,Z_T, W_T) \in \Lambda_T \times \hat{\Lambda}_T,
\end{cases}
\end{align}
and the upper PHJI equation
\begin{align}
\label{eq_5_3}
\begin{cases}
\mathbb{H}^+(A_t,Z_t, W_t, (\partial_t \mathbb{U}^{u,v}, \mathbb{U}, \partial_x \mathbb{U}, \partial_{xx} \mathbb{U})(A_t;Z_t, W_t) ) \\
: = \inf_{u \in \mathrm{U}} \sup_{v \in \mathrm{V}} \Bigl \{ \partial_t \mathbb{U}(A_t;Z_t^{u}, W_t^{v}) \\
~~~~~~~~~~ + \mathcal{H}(A_t,Z_t^u, W_t^v,(\mathbb{U}, \partial_x \mathbb{U}, \partial_{xx} \mathbb{U})(A_t;Z_t, W_t) ) \Bigr \} = 0,\\
~~~~~~~~~~~~ t \in [0,T),~ (A_t, Z_t, W_t) \in \Lambda \times \hat{\Lambda} \\
	\mathbb{U}(A_T;Z_T, W_T) = m(A_T),~ (A_T, Z_T, W_T) \in \Lambda_T \times \hat{\Lambda}_T.
\end{cases}	
\end{align}

\begin{remark}\label{Remark_5_1}
\begin{enumerate}[(1)]
	\item In (\ref{eq_5_2}), $\partial_t \mathbb{L}^{u,v}(A_t;Z_t, W_t) := \partial_t \mathbb{L}(A_t;Z_t^u, W_t^u)$. From Section \ref{Section_1_1}, the time derivative of $\mathbb{L}$ in $(A_t,Z_t,W_t)$ (if it exists) can be written as follows:
	\begin{align*}
	\partial_t \mathbb{L}(A_t;Z_t^u, W_t^u) = \lim_{\delta t \downarrow 0} \frac{\mathbb{L}(A_{t,\delta t}; Z_{t,\delta t}, W_{t,\delta t}) - \mathbb{L}(A_t;Z_t,W_t)}{\delta t},
	\end{align*}
	where $(Z_t^u,W_t^v)$ is induced due to the definition of $\otimes$ in (\ref{eq_2_4}) (see also \cite{Saporito_SICON_2019}). The space derivative of $\mathbb{L}$ with respect to $A_t$ is given by $		\partial_x \mathbb{L}(A_t;Z_t,W_t) = \begin{bmatrix}
 				\partial_x^1 \mathbb{L}(A_t;Z_t,W_t) & \cdots & \partial_x^n \mathbb{L}(A_t;Z_t,W_t)
 			\end{bmatrix}^\top$ (if it exists), where
	\begin{align*}	
		\partial_x^i \mathbb{L}(A_t;Z_t,W_t) &= \lim_{h \downarrow 0} \frac{\mathbb{L}(A_{t}^{(he_i)}; Z_{t}, W_{t}) - \mathbb{L}(A_t;Z_t,W_t)}{h}.
	\end{align*}
	Note that by the definition in  (\ref{eq_2_4}), we have $\mathbb{L}(A_t;Z_t,W_t) = \mathbb{L}(A_t;Z_t^{(he_i)},W_{t})$ and $\mathbb{L}(A_t;Z_t,W_t) = \mathbb{L}(A_t;Z_t,W_{t}^{(he_i)})$, which implies that $\mathbb{L}$ satisfies the \emph{predictable dependence} condition in the sense of \cite{Cont_JFA_2010}. Hence, the space derivative of $\mathbb{L}$ with respect to the control of the players is zero; see \cite[Remark 4]{Cont_JFA_2010} and \cite[Remark 2.3]{Saporito_SICON_2019}. The same argument applies to (\ref{eq_5_3}).
	\item If there is a functional in $\mathcal{C}^{1,2}(\Lambda; \hat{\Lambda})$ in the sense of Definition \ref{Definition_0_2}, which solves (\ref{eq_5_2}), then it is a \emph{classical solution} of (\ref{eq_5_2}). Also, similar to \cite{Ekren_AP_2014, Peng_Gexpectation_arXiv_2011}, the classical sub-solution (respectively, super-solution) is  defined if the \enquote{$=0$} in (\ref{eq_5_2}) is replaced by \enquote{$\geq 0$} (respectively, \enquote{$\leq 0$}). When there is a classical solution of (\ref{eq_5_2}), it means that it is both classical sub- and super-solutions. The same argument can be applied to the upper PHJI equation in (\ref{eq_5_3}). 
\end{enumerate}
\end{remark}

%The next remark states some special cases of the PHJI equations in view of Remark \ref{Remark_3_7}.
\begin{remark}\label{Remark_5_1_1}
For the state path-dependence case (see Remark \ref{Remark_3_7}), (\ref{eq_5_2}) and (\ref{eq_5_3}) are reduced to the state path-dependent HJI equations in (\ref{eq_6_1}) and (\ref{eq_6_2}). In addition, in the Markovian formulation (see Remark \ref{Remark_3_7}), the (lower and upper) PHJI equations are equivalent to those in \cite[(4.1) and (4.2)]{Buckdahn_SICON_2008}
%\begin{enumerate}[(1)]
%	\item For the state path-dependence case (Remark \ref{Remark_3_7}), (\ref{eq_5_2}) and (\ref{eq_5_3}) are reduced to the state path-dependent HJI equations in (\ref{eq_6_1}) and (\ref{eq_6_2}).
%	\item In addition, for the Markovian formulation (see Remark \ref{Remark_3_7}), the (lower and upper) PHJI equations are equivalent to those in \cite[(4.1) and (4.2)]{Buckdahn_SICON_2008}.
%\end{enumerate}	
\end{remark}

We fix $\kappa \in (0,\frac{1}{2})$ in the $\kappa$-H\"older modulus. The notion of the viscosity solution is given belowas follows, which was first introduced in \cite{Tang_DCD_2015} for the state path-dependent case.

\begin{definition}\label{Definition_5_2}
	\begin{enumerate}[(i)]
		\item A real-valued functional $\mathbb{L} \in \mathcal{C}(\Lambda; \hat{\Lambda})$ is said to be a viscosity sub-solution of the lower PHJI equation in (\ref{eq_5_2}) if for $(A_T,Z_T,W_T) \in \Lambda_T \times \hat{\Lambda}_T$ and $\mu, \mu_0 > 0$, $\mathbb{L}(A_T; Z_T,W_T) \leq m(A_T)$ and for all test functions $\phi \in \mathcal{C}^{1,2}_{\kappa}(\Lambda; \hat{\Lambda})$ satisfying the predictable dependence in the sense of \cite{Cont_JFA_2010}, i.e., $\phi(A_t;Z_t,W_t) = \phi(A_t;Z_{t-},W_{t-})$ and $0 = (\mathbb{L} - \phi)(\bar{A}_t;Z_t,W_t) = \sup_{A_s \in \mathbb{C}^{\kappa,\mu,\mu_0}} (\mathbb{L} - \phi)(A_s;Z_t,W_t)$, 
%		\begin{align*}
%		0 = (\mathbb{L} - \phi)(\bar{A}_t;Z_t,W_t) = \sup_{A_s \in \mathbb{C}^{\kappa,\mu,\mu_0}} (\mathbb{L} - \phi)(A_s;Z_t,W_t),
%		\end{align*}
		where $\bar{A}_t \in \mathbb{C}^{\kappa,\mu,\mu_0}$, the following inequality holds:
		\begin{align*}
			& \liminf_{\mu \rightarrow \infty}  \mathbb{H}^-(\bar{A}_t,Z_t, W_t, (\partial_t \phi^{u,v}, \phi, \partial_x \phi, \partial_{xx} \phi)(\bar{A}_t;Z_t, W_t) )  \geq 0.
		\end{align*}
		\item A real-valued functional $\mathbb{L} \in \mathcal{C}(\Lambda; \hat{\Lambda})$ is said to be a viscosity super-solution of the lower PHJI equation in (\ref{eq_5_2}) if for $(A_T,Z_T,W_T) \in \Lambda_T \times \hat{\Lambda}_T$ and $\mu, \mu_0 > 0$, $\mathbb{L}(A_T; Z_T,W_T) \geq m(A_T)$ and for all test functions $\phi \in \mathcal{C}^{1,2}_{\kappa}(\Lambda; \hat{\Lambda})$ satisfying  the predictable dependence in the sense of \cite{Cont_JFA_2010}, i.e., $\phi(A_t;Z_t,W_t) = \phi(A_t;Z_{t-},W_{t-})$ and $0 = (\mathbb{L} - \phi)(\bar{A}_t; Z_t,W_t) = \inf_{A_s \in \mathbb{C}^{\kappa,\mu,\mu_0}} (\mathbb{L} - \phi)(A_s; Z_t,W_t)$, 
%		\begin{align*}
%		0 = (\mathbb{L} - \phi)(\bar{A}_t; Z_t,W_t) = \inf_{A_s \in \mathbb{C}^{\kappa,\mu,\mu_0}} (\mathbb{L} - \phi)(A_s; Z_t,W_t),
%		\end{align*}
		where $\bar{A}_t \in \mathbb{C}^{\kappa,\mu,\mu_0}$, the following inequality holds:
		\begin{align*}
			& \limsup_{\mu \rightarrow \infty}  \mathbb{H}^-(\bar{A}_t,Z_t, W_t, (\partial_t \phi^{u,v}, \phi, \partial_x \phi, \partial_{xx} \phi)(\bar{A}_t;Z_t, W_t) ) \leq 0.
		\end{align*}
	\item A real-valued functional $\mathbb{L} \in \mathcal{C}(\Lambda; \hat{\Lambda})$ is said to be a viscosity solution if it is both a viscosity sub-solution and super-solution of (\ref{eq_5_2}).
	\item The viscosity sub-solution, super-solution and solution of the upper PHJI equation in (\ref{eq_5_3}) are defined in similar ways.
	\end{enumerate}
\end{definition}

\begin{remark}\label{Remark_5_3}
\begin{enumerate}[(1)]
	\item In Definition \ref{Definition_5_2}, in view of Remark \ref{Remark_5_1}, $\partial_t \phi^{u,v}(\bar{A}_t;Z_t,W_t) := \partial_t \phi(\bar{A}_t;Z_t^u,W_t^u)$. 
	For the Markovian case, Definition \ref{Definition_5_2} is equivalent to that of the classical one in \cite{Buckdahn_SICON_2008, Crandall_AMS_1992_Viscosity, Touzi_Book}. Moreover, we can easily check that if the viscosity solution of (\ref{eq_5_2}) further belongs to $\mathcal{C}_{\kappa}^{1,2}(\Lambda; \hat{\Lambda})$ satisfying the predictable dependence, then it is also the classical solution of (\ref{eq_5_2}). The same argument applies to (\ref{eq_5_3}). This implies that when the (lower and upper) value functionals are in $\mathcal{C}_{\kappa}^{1,2}(\Lambda; \hat{\Lambda})$, they are classical solutions of the (lower and upper) PHJI equations.
	\item The definition of viscosity solutions in Definition \ref{Definition_5_2} is different from that in \cite{Ekren_AP_2014, Ekren_AP_2016_Part_I, Ekren_AP_2016}, which was applied to SZSDGs in weak formulation in \cite{Pham_SICON_2014, Possamai_arXiv_2018}. In particular, in \cite{Ekren_AP_2014, Ekren_AP_2016_Part_I, Ekren_AP_2016}, a nonlinear expectation was included in the corresponding semi-jets, which is closely related to a certain class of BSDEs via Feynman-Kac formula. It is interesting to investigate the relationship (or equivalence) between Definition \ref{Definition_5_2} and the definition in \cite{Ekren_AP_2014, Ekren_AP_2016_Part_I, Ekren_AP_2016}. As noted in Section \ref{Section_1} and \cite[Remark 3.7 and p. 10]{Possamai_arXiv_2018}, the general uniqueness in the sense of \cite{Ekren_AP_2014, Ekren_AP_2016_Part_I, Ekren_AP_2016} has not been completely solved, and the SZSDG in weak formulation requires more stringent assumptions on the coefficients than the SZSDG in strong formulation. Since we consider the SZSDG in strong formulation, we use the notion of viscosity solutions in \cite{Tang_DCD_2015}, which was applied to the state path-dependent (one-player) stochastic control problem in strong formulation. A similar definition was also introduced in \cite{Qui_SICON_2018} to study a class of stochastic HJB equations (in strong formulation).
	\end{enumerate}
\end{remark}

\begin{remark}\label{Remark_5_5}
This remark will be used in the proof of Theorem \ref{Theorem_5_4} given below. The condition of the predictable dependence for the test function $\phi$ in Definition \ref{Definition_5_2} is introduced due to the control path-dependent nature of the SZSDG with (\ref{eq_2_4}). Specifically, from the predictable dependence of $\phi$ with resect to the control of the players in the sense of \cite{Cont_JFA_2010}, i.e., $\phi(A_t;Z_t,W_t) = \phi(A_t;Z_{t-},W_{t-})$, and the definition in (\ref{eq_5_1_1_1_1_1_1_1}) and (\ref{eq_2_4}), it holds that $\phi(A_t;Z_t,W_t) = \phi(A_t;Z_t^{(he_i)},W_{t})$ and $\phi(A_t;Z_t,W_t) = \phi(A_t;Z_t,W_{t}^{(he_i)})$ (see also (1) of Remark \ref{Remark_5_1}). Therefore, the (space) derivative of $\phi$ with respect to the control of the players is zero, i.e., $\partial_{u} \phi = 0$, $\partial_{uu}\phi = 0$, $\partial_{v} \phi = 0$ and $\partial_{vv} \phi = 0$. Similar discussions can be found in \cite[Remark 4]{Cont_JFA_2010} and \cite[Remark 2.3]{Saporito_SICON_2019}. We should also mention that for the state path-dependent case (see Remarks \ref{Remark_3_7} and \ref{Remark_5_1_1}), the predictable dependence condition is not needed.
\end{remark}

We state the main result of this section.
\begin{theorem}\label{Theorem_5_4}
Suppose that Assumption \ref{Assumption_1} holds. Then the lower value functional in (\ref{eq_2_6}) is the viscosity solution of the lower PHJI equation in (\ref{eq_5_2}). The upper value functional in (\ref{eq_2_7}) is the viscosity solution of the upper PHJI equation in (\ref{eq_5_3}).	
\end{theorem}

Before proving the theorem, for $\mu > 0$, $\epsilon \in (0,\mu)$, $r \in [0,t]$ and $A_t \in \mathbb{C}^{\kappa,\mu,\mu_0}$, let $A_t^\epsilon$ be the perturbed version of $A_t$ defined by
\begin{align*}	
a_r^\epsilon := \begin{cases}
	a_r, & \text{if}~ |a_r - a_t| \leq (\mu - \epsilon)|r-t|^{\kappa} \\
	a_t + (\mu - \epsilon)(t-r)^\kappa \frac{a_r - a_t}{|a_r - a_t|}, & \text{if}~ |a_r - a_t| \geq (\mu - \epsilon)|r-t|^{\kappa}. 
\end{cases}
\end{align*}
Note that $A_t^\epsilon := \{a_r^\epsilon,~ r \in [0,t]\}$. The perturbation is essential to prove Theorem \ref{Theorem_5_4}; see \cite[Remark 6]{Tang_DCD_2015}.

We state the following lemma, whose proof is given in \cite[Lemma 5.1]{Tang_DCD_2015}.

\begin{lemma}\label{Lemma_5_5}
Let $\mu,\mu_0 > 0$. Assume that $[\![ A_t ]\!]_{\kappa} \leq \mu$, $\|A_t\|_{\infty} \leq \mu_0$, i.e., $A_t \in \mathbb{C}^{\kappa,\mu,\mu_0}$, and $\epsilon \in (0,\frac{1}{2}\mu]$. Then we have
	\begin{enumerate}[(i)]
		\item $\|A_t^\epsilon - A_t \|_\infty \leq 2 \mu_0 \epsilon (\mu - \epsilon)^{-1} \leq 4 \mu_0 \epsilon \mu^{-1}$.
		\item $[\![ A_t^\epsilon ]\!]_{\kappa} \leq \mu$.
		\item There exists a constant $C > 0$, independent of $\mu$, such that for any $d$ with $\frac{d}{2}(\frac{1}{2} - \kappa) > 1$ and $t + \tau < T$, $\mathbb{P}  ( 	[\![ X_{t+\tau}^{t,A_t^\epsilon;Z_t \otimes u, W_t \otimes v} ]\!]_{\kappa} > \mu  ) \leq C \tau ^{d(\frac{1}{2} - \kappa)} \epsilon^{-d}$.
%		\begin{align*}
%		\mathbb{P} \Bigl ( 	[\![ X_{t+\tau}^{t,A_t^\epsilon;Z_t \otimes u, W_t \otimes v} ]\!]_{\kappa} > \mu \Bigr ) \leq C \tau ^{d(\frac{1}{2} - \kappa)} \epsilon^{-d}.
%		\end{align*}
	\end{enumerate}
\end{lemma}

The proof of Theorem \ref{Theorem_5_4} is given as follows.

\begin{proof}[Proof of Theorem \ref{Theorem_5_4}]
We first prove that the lower value functional in (\ref{eq_2_6}) is the viscosity super-solution of the lower PHJI equation in (\ref{eq_5_2}). Note that in view of Lemmas \ref{Lemma_2} and \ref{Lemma_3}, it is clear that $\mathbb{L} \in \mathcal{C}(\Lambda;\hat{\Lambda})$. Furthermore, from (\ref{eq_2_6}), we have $\mathbb{L}(A_T;Z_T,W_T) \geq  m(A_T)$.

From the definition of the viscosity super-solution in (ii) of Definition \ref{Definition_5_2} and Lemma \ref{Lemma_0_5}, for $\phi \in \mathcal{C}^{1,2}_{\kappa}(\Lambda; \hat{\Lambda} )$, $\mu > 1$ and $\mu_0 > 0$, 
\begin{align*}
		0 = (\mathbb{L} - \phi)(\bar{A}_t; Z_t,W_t) = \inf_{A_s \in \mathbb{C}^{\kappa,\mu,\mu_0}} (\mathbb{L} - \phi)(A_s; Z_t,W_t),
\end{align*}
where $\bar{A}_t \in \mathbb{C}^{\kappa,\mu,\mu_0}$. Let $\delta := \mu_0 - \|\bar{A}_t\|_\infty > 0$. 

For any $\epsilon \in (0,\frac{1}{2} \wedge \frac{\delta \mu}{8 \mu_0})$, in view of (i) and (ii) in Lemma \ref{Lemma_5_5}, we have $\|\bar{A}_t^\epsilon - \bar{A}_t \|_{\infty} \leq 4 \mu_0 \epsilon \mu^{-1} < \frac{\delta}{2}$. Consider the following $\mathbb{F}$-stopping time:
\begin{align*}
\xi^\epsilon & := \inf \{r > t~:~	[\![ X_r^{t,\bar{A}_t^\epsilon; Z_t \otimes \alpha(v), W_t \otimes v } ]\!]_{\kappa} > \mu \} \\
&~~~ \wedge \inf\{ r > t~:~ \| X_r^{t,\bar{A}_t^\epsilon; Z_t \otimes \alpha(v), W_t \otimes v } \|_\infty > \mu_0 \}.
\end{align*}
By definition, for any $s < \xi^\epsilon$,  $X_{s}^{t,\bar{A}_t^\epsilon; Z_t \otimes \alpha(v), W_t \otimes v } \in \mathbb{C}^{\kappa,\mu,\mu_0}$ and for a small $\tau$ with $t+\tau \leq T$,
\begin{align*}
	\{\xi^{\epsilon} \leq t+\tau \} \subset \{ [\![ X_{t+\tau}^{t,\bar{A}_t^\epsilon; Z_t \otimes \alpha(v), W_t \otimes v} ] \! ]_{\kappa} > \mu \} \cup \{ \| X_{t+\tau}^{t,\bar{A}_t^\epsilon; Z_t \otimes \alpha(v), W_t \otimes v} \|_{\infty}  > \mu_0 \}.
\end{align*}
Hence, from (iii) of Lemma \ref{Lemma_5_5}, we have
\begin{align*}
\mathbb{P} \Bigl ( [\![ X_{t+\tau}^{t,\bar{A}_t^\epsilon; Z_t \otimes \alpha(v), W_t \otimes v} ] \! ]_{\kappa} > \mu \Bigr ) \leq C \tau^{d(\frac{1}{2} - \kappa)} \epsilon^{-d},	
\end{align*}
and by (ii) of Lemma \ref{Lemma_1} and Markov inequality, 
\begin{align*}
	\mathbb{P} \Bigl (\| X_{t+\tau}^{t,\bar{A}_t^\epsilon; Z_t \otimes \alpha(v), W_t \otimes v} \|_{\infty}  > \mu_0  \Bigr ) & \leq 	\mathbb{P} \Bigl (\| X_{t+\tau}^{t,\bar{A}_t^\epsilon; Z_t \otimes \alpha(v), W_t \otimes v} - \bar{A}_{t,t+\tau}^\epsilon \|_{\infty}  > \frac{\delta}{2}  \Bigr ) \\
	  & \leq C (1 + \mu_0^6)\tau^3/\delta^{6}.
\end{align*}
This implies that
\begin{align}
\label{eq_5_4}
\mathbb{P}(	\xi^{\epsilon} \leq t+\tau) \leq C \tau^{d(\frac{1}{2} - \kappa)} \epsilon^{-d} + C (1 + \mu_0^6)\tau^3/\delta^{6} ~\downarrow 0~\text{as $\tau \downarrow 0$}.
\end{align}

Now, from the dynamic programming principle in (\ref{eq_3_3}) of Theorem \ref{Theorem_1}, 
\begin{align}
\label{eq_5_4_1}
& \mathbb{L}(\bar{A}_t^\epsilon; Z_t, W_t) - \phi(\bar{A}_t^\epsilon; Z_t, W_t)\\
\nonumber & = \essinf_{\alpha \in \mathcal{A}[t,t+\tau]} \esssup_{v \in \mathcal{V}[t,t+\tau]} \Pi_{t}^{t,t+\tau,\bar{A}_t^\epsilon; Z_{t} \otimes \alpha(v), W_{t} \otimes v}  \\
\nonumber & ~~~ \Bigl [ \mathbb{L}(X_{t+\tau}^{t,\bar{A}_t^\epsilon; Z_{t} \otimes \alpha(v), W_t \otimes v}; (Z_t \otimes \alpha(v))_{t+\tau}, (W_t \otimes v)_{t+\tau} ) \Bigr ] - \phi(\bar{A}_t^\epsilon; Z_t, W_t).
\end{align}
Note also that
\begin{align*}
& \mathbb{L}(\bar{A}_t^\epsilon; Z_t, W_t) \geq \esssup_{v \in \mathcal{V}[t,t+\tau]}  \essinf_{u \in \mathcal{U}[t,t+\tau]}  \Pi_{t}^{t,t+\tau, \bar{A}_t^\epsilon; Z_{t} \otimes u, W_{t} \otimes v} 	\\
&~~~~~~~~~~ \Bigl [ \mathbb{L}(X_{t+\tau}^{t, \bar{A}_t^\epsilon; Z_{t} \otimes u, W_t \otimes v}; (Z_t \otimes u)_{t+\tau}, (W_t \otimes v)_{t+\tau} ) \Bigr ].
\end{align*}
Then similar to (\ref{eq_4_13_1_1}), for any $\epsilon^\prime > 0$, there exists $u^{\epsilon^\prime} \in \mathcal{U}[t,t+\tau]$ such that for any $v \in \mathcal{V}[t,t+\tau]$,
\begin{align}
\label{eq_5_4_2}
& \mathbb{L}(\bar{A}_t^\epsilon; Z_t, W_t) \geq \Pi_{t}^{t,t+\tau, \bar{A}_t^\epsilon; Z_{t} \otimes u^{\epsilon^\prime}, W_{t} \otimes v}	 \\
&~~~~~~~~~~ \Bigl [ \mathbb{L}(X_{t+\tau}^{t, \bar{A}_t^\epsilon; Z_{t} \otimes u^{\epsilon^\prime}, W_t \otimes v}; (Z_t \otimes u^{\epsilon^\prime})_{t+\tau}, (W_t \otimes v)_{t+\tau} ) \Bigr ] - \epsilon^\prime \tau, \nonumber
\end{align}
where in view of the definition of $\Pi$, $\Pi$ in the above expression can be rewritten as (superscript $t+\tau$ is omitted)
\begin{align}	
\label{eq_5_5}
\begin{cases}
\dd y_s^{t,\bar{A}_t^\epsilon;Z_t \otimes u^{\epsilon^\prime},W_t \otimes v} = - l(X_s^{t,\bar{A}_t^\epsilon; Z_t \otimes u^{\epsilon^\prime}, W_t \otimes v},y_s^{t, \bar{A}_t^\epsilon; Z_t \otimes u^{\epsilon^\prime}, W_t \otimes v}, \\
~~~~~~~~~~~~~~ q_s^{t, \bar{A}_t^\epsilon; Z_t \otimes u^{\epsilon^\prime}, W_t \otimes v},  (Z_t \otimes u^{\epsilon^\prime})_s, (W_t \otimes v)_s) \dd s \\
~~~~~~~~~~ + q_s^{t, \bar{A}_t^\epsilon; Z_t \otimes u^{\epsilon^\prime}, W_t \otimes v} \dd B_s,~ s \in [t,t+\tau) \\
y_{t+\tau}^{t,\bar{A}_t^\epsilon; Z_t \otimes u^{\epsilon^\prime},W_t \otimes v}  = \mathbb{L}(X_{t+\tau}^{t,\bar{A}_t^\epsilon; Z_{t} \otimes u^{\epsilon^\prime}, W_t \otimes v}; (Z_t \otimes u^{\epsilon^\prime})_{t+\tau}, (W_t \otimes v)_{t+\tau} ).
\end{cases}
\end{align}
%which is equivalent to for $s \in [t,t+\tau]$,\begin{align*}
%	y_s^{t,\bar{A}_t^\epsilon; Z_t \otimes u^{\epsilon^\prime}, W_t \otimes v} &= \Pi_{s}^{t,t+\tau, \bar{A}_t^\epsilon; Z_{t} \otimes u^{\epsilon^\prime}, W_{t} \otimes v}	 \\
%&~~~~~ \Bigl [ \mathbb{L}(X_{t+\tau}^{t, \bar{A}_t^\epsilon; Z_{t} \otimes u^{\epsilon^\prime}, W_t \otimes v}; (Z_t \otimes u^{\epsilon^\prime})_{t+\tau}, (W_t \otimes v)_{t+\tau} ) \Bigr ].
%\end{align*}

On the other hand, by using the functional It\^o formula in Lemma \ref{Lemma_0_7}, we have
\begin{align}
\label{eq_5_6}
& \phi(X_s^{t,\bar{A}_t^\epsilon; Z_t \otimes u^{\epsilon^\prime}, W_t \otimes v}; (Z_t \otimes u^{\epsilon^\prime})_s,(W_t \otimes v)_s) 
\\
& = \phi(\bar{A}_t^\epsilon; Z_t, W_t) + \int_t^{s} F(X_r^{t, \bar{A}_t^\epsilon; Z_t \otimes u^{\epsilon^\prime},W_t \otimes v}, (Z_t \otimes u^{\epsilon^\prime})_r, (W_t \otimes v)_r)	 \dd r \nonumber \\
&~~~ - \int_t^s l(X_r^{t, \bar{A}_t^\epsilon; Z_t \otimes u^{\epsilon^\prime}, W_t \otimes v}, \phi(X_r^{t,\bar{A}_t^\epsilon; Z_t \otimes u^{\epsilon^\prime}, W_t \otimes v}; (Z_t \otimes u^{\epsilon^\prime})_r, (W_t \otimes v)_r), \nonumber \\
&~~~~~~~~~~  \langle \partial_x \phi (X_r^{t,\bar{A}_t^\epsilon; Z_t \otimes u^{\epsilon^\prime},W_t \otimes v}; (Z_t \otimes u^{\epsilon^\prime})_r, (W_t \otimes v)_r), \nonumber \\
&~~~~~~~~~~ \sigma(X_r^{t, \bar{A}_t^\epsilon; Z_t \otimes u^{\epsilon^\prime}, W_t \otimes v},(Z_t \otimes u^{\epsilon^\prime} )_r, (W_t \otimes v)_r) \rangle, (Z_t \otimes u^{\epsilon^\prime})_r, (W_t \otimes v)_r) \dd r \nonumber \\
&~~~ + \int_t^{s} \langle \partial_x \phi (X_r^{t,\bar{A}_t^\epsilon; Z_t \otimes u^{\epsilon^\prime}, W_t \otimes v}; (Z_t \otimes u^{\epsilon^\prime} )_r, (W_t \otimes v)_r), \nonumber \\
&~~~~~~~~~~ \sigma(X_r^{t, \bar{A}_t^\epsilon; Z_t \otimes u^{\epsilon^\prime}, W_t \otimes v},(Z_t \otimes u^{\epsilon^\prime} )_r, (W_t \otimes v)_r) \rangle \dd B_r, \nonumber
\end{align}
where
\begin{align*}
F(A_t,Z_t^u, W_t^v)	&= \partial_{t} \phi(A_t; Z_t^u, W_t^v)  + \frac{1}{2} \Tr (\partial_{xx}\phi(A_t; Z_t, W_t) \sigma \sigma^\top(A_t, Z_t^{u}, W_t^{v})) \\
&~~~ + \langle \partial_x \phi (A_t; Z_t, W_t), f(A_t, Z_t^{u}, W_t^{v}) \rangle \\
&~~~ + l(A_t, \phi(A_t; Z_t, W_t), \langle \partial_x \phi (A_t; Z_t, W_t), \sigma(A_t,Z_t^{u}, W_t^{v}) \rangle, Z_t^{u}, W_t^{v}).
\end{align*}
Here, we used the fact that the (space) derivative of $\phi$ with respect to the control of the players is zero as stated in Remark \ref{Remark_5_5}.

Let
\begin{align}
\label{eq_5_7}
\bar{y}_s^{t,\bar{A}_t^\epsilon;Z_t \otimes u^{\epsilon^\prime}, W_t^ \otimes v} & := \phi(X_s^{t,\bar{A}_t^\epsilon; Z_t \otimes u^{\epsilon^\prime}, W_t \otimes v}; (Z_t \otimes u^\epsilon)_s,(W_t \otimes v)_s) \\
&~~~ - y_s^{t,\bar{A}_t^\epsilon; Z_t \otimes u^{\epsilon^\prime}, W_t \otimes v} \nonumber
%\\
\end{align}
\begin{align}
\label{eq_5_7_1_1}
\bar{q}_s^{t, \bar{A}_t^\epsilon; Z_t \otimes u^{\epsilon^\prime}, W_t \otimes v} & := \langle \partial_x \phi (X_r^{t,\bar{A}_t^\epsilon; Z_t \otimes u^{\epsilon^\prime}, W_t \otimes v}; (Z_t \otimes u^{\epsilon^\prime} )_r, (W_t \otimes v)_r), \\
&~~~~~~~~~~ \sigma(X_r^{t, \bar{A}_t^\epsilon; Z_t \otimes u^{\epsilon^\prime}, W_t \otimes v},(Z_t \otimes u^{\epsilon^\prime} )_r, (W_t \otimes v)_r) \rangle  \nonumber \\
&~~~ - q_s^{t, \bar{A}_t^\epsilon; Z_t \otimes u^{\epsilon^\prime}, W_t \otimes v}. \nonumber
\end{align}
From (\ref{eq_5_5}) and (\ref{eq_5_6}), 
\begin{align}
\label{eq_5_8}
& \dd \bar{y}_s^{t,\bar{A}_t^\epsilon;Z_t \otimes u^{\epsilon^\prime}, W_t \otimes v} \\
& = F(X_s^{t,\bar{A}_t^\epsilon; Z_t \otimes u^{\epsilon^\prime}, W_t \otimes v}, (Z_t \otimes u^{\epsilon^\prime} )_s, (W_t \otimes v)_s)	 \dd s + \bigl [ - H_s	\bar{y}_s^{t,\bar{A}_t^\epsilon; Z_t \otimes u^{\epsilon^\prime}, W_t \otimes v} \nonumber  \\
&~~~  - \langle \bar{H}_s, \bar{q}_s^{t, \bar{A}_t^\epsilon;  Z_t \otimes u^{\epsilon^\prime}, W_t \otimes v} \rangle \bigr ] \dd s  + \bar{q}_s^{t, \bar{A}_t^\epsilon; Z_t \otimes u^{\epsilon^\prime}, W_t \otimes v} \dd B_s,  \nonumber
\end{align}
where $|H| \leq C$ and $|\bar{H}| \leq C$ due to Assumption \ref{Assumption_1}. We have from (\ref{eq_5_7}),
\begin{align*}
\bar{y}_{t+\tau}^{t,\bar{A}_t^\epsilon; Z_t \otimes u^{\epsilon^\prime}, W_t \otimes v} & = \phi(X_{t+\tau}^{t,\bar{A}_t^\epsilon; Z_t \otimes u^{\epsilon^\prime}, W_t \otimes v}; (Z_t \otimes u^{\epsilon^\prime} )_{t+\tau},(W_t \otimes v)_{t+\tau}) \\
&~~~ - \mathbb{L}(X_{t+\tau}^{t, \bar{A}_t^\epsilon; Z_{t} \otimes u^{\epsilon^\prime}, W_t \otimes v}; (Z_t \otimes u^{\epsilon^\prime})_{t+\tau}, (W_t \otimes v)_{t+\tau} ).
\end{align*}

Notice that (\ref{eq_5_8}) is a linear BSDE; hence, by using Lemma \ref{Lemma_0_7} and \cite[Proposition 4.1.2]{Zhang_book_2017}, its explicit unique solution can be written as follows:
\begin{align}
\label{eq_5_11}
& \bar{y}_t^{t,\bar{A}_t^\epsilon; Z_t \otimes u^{\epsilon^\prime}, W_t \otimes v}  = \mathbb{E} \Bigl [ \bar{y}_{t+\tau}^{t,\bar{A}_t^\epsilon; Z_t \otimes u^{\epsilon^\prime}, W_t \otimes v} \Phi_{t+\tau} \\
&~~~~~~~~~~ - \int_t^{t+\tau} \Phi_{r} F(X_r^{t,\bar{A}_t^\epsilon; Z_t \otimes u^{\epsilon^\prime}, W_t \otimes v}, (Z_t \otimes u^{\epsilon^\prime} )_r, (W_t \otimes v)_r) \dd r \bigl | \mathcal{F}_t \Bigr ], \nonumber
\end{align}
where $\Phi$ is the scalar-valued state transition process given by $\dd \Phi_r =   \Phi_r H_r\dd r + \Phi_r \bar{H}_r \dd B_r$, $r \in (t,t+\tau]$, with $\Phi_{t} = 1$, i.e., $\Phi_r = \exp( \int_t^r \bar{H}_s \dd B_s + \int_t^r [ H_s - \frac{1}{2} |\bar{H}_s|^2] \dd s )$.
%\begin{align*}
%\Phi_r = \exp \Bigl ( \int_t^r \bar{H}_s \dd B_s + \int_t^r [ H_s - \frac{1}{2} |\bar{H}_s|^2] \dd s \Bigr ),~ r \in [t,t+\tau],
%\end{align*}
%that is equivalent to 
%\begin{align*}
%\begin{cases}
%\dd \Phi_r =   \Phi_r H_r\dd r + \Phi_r \bar{H}_r \dd B_r,~ r \in (t,t+\tau]\\
%\Phi_{t} = 1.
%\end{cases}	
%\end{align*}

From (\ref{eq_5_4_1}) and (\ref{eq_5_4_2}),  together with (\ref{eq_5_11}) and the predictable dependence of $\phi$, 
\begin{align}
\label{eq_5_11_1}
&\mathbb{L}(\bar{A}_t^\epsilon; Z_t, W_t) - \phi(\bar{A}_t^\epsilon; Z_t, W_t) + \epsilon^\prime \tau \\
& \geq \mathbb{E} \Bigl [  - \bar{y}_{t+\tau}^{t,\bar{A}_t^\epsilon; Z_t \otimes u^{\epsilon^\prime}, W_t \otimes v} \Phi_{t+\tau} \nonumber \\
&~~~~~~~~ + \int_t^{t+\tau} \Phi_{r} F(X_r^{t,\bar{A}_t^\epsilon; Z_t \otimes u^{\epsilon^\prime}, W_t \otimes v}, (Z_t \otimes u^{\epsilon^\prime} )_r, (W_t \otimes v)_r) \dd r \bigl | \mathcal{F}_t \Bigr ] \nonumber \\
& = \mathbb{E} \Bigl [ \int_t^{t+\tau} F(\bar{A}_t, (Z_t \otimes u^{\epsilon^\prime} )_r, (W_t \otimes v)_r) \dd r | \mathcal{F}_t \Bigr ] + L^{(1)} + L^{(2)} + L^{(3)}, \nonumber
\end{align}
where
\begin{align*}
L^{(1)} & := 	- \mathbb{E} \Bigl [  \bar{y}_{t+\tau}^{t,\bar{A}_t^\epsilon; Z_t \otimes u^{\epsilon^\prime}, W_t \otimes v} \Phi_{t+\tau} | \mathcal{F}_t \Bigr ] \\
L^{(2)} & := 	\mathbb{E} \Bigl [ \int_t^{t+\tau} F(X_r^{t,\bar{A}_t^\epsilon; Z_t \otimes u^{\epsilon^\prime}, W_t \otimes v}, (Z_t \otimes u^{\epsilon^\prime})_r, (W_t \otimes v)_r) \dd r  \\
&~~~~~~~~~~ - \int_t^{t+\tau} F(\bar{A}_t, (Z_t \otimes u^{\epsilon^\prime} )_r, (W_t \otimes v)_r) \dd r | \mathcal{F}_t \Bigr ] \\
L^{(3)} & := 	\mathbb{E} \Bigl [ \int_t^{t+\tau} \Phi_{r} F(X_r^{t,\bar{A}_t^\epsilon; Z_t \otimes u^{\epsilon^\prime}, W_t \otimes v}, (Z_t \otimes u^{\epsilon^\prime} )_r, (W_t \otimes v)_r) \dd r  \\
&~~~~~~~~~~ -  \int_t^{t+\tau} F(X_r^{t,\bar{A}_t^\epsilon; Z_t \otimes u^{\epsilon^\prime}, W_t \otimes v}, (Z_t \otimes u^{\epsilon^\prime})_r, (W_t \otimes v)_r) \dd r | \mathcal{F}_t \Bigr ].
\end{align*}

In view of (ii) in Lemma \ref{Lemma_1} and the fact that $\Phi$ is the linear SDE, we have
\begin{align*}
\mathbb{E} \bigl [ \|	X_{t+\tau}^{t,\bar{A}_t^\epsilon; Z_t \otimes u^{\epsilon^\prime}, W_t \otimes v} - \bar{A}_{t,\tau}^\epsilon \|_{\infty}^2 \bigl | \mathcal{F}_t \bigr ] & \leq C \tau,~ \mathbb{E} \bigl [ \sup_{r \in [t,t+\tau]} | \Phi_r - 1|^2 | \mathcal{F}_t \bigr ] \leq C \tau.
\end{align*}
Furthermore, due to the property of $\phi \in \mathcal{C}^{1,2}_{\kappa}(\Lambda; \hat{\Lambda} )$ and Assumption \ref{Assumption_1}, for $t_1,t_2 \in [t,T]$ and $A_{t_1}, A_{t_2} \in \Lambda$,
\begin{align*}
 |\phi(A_{t_1}^1; Z, W) - \phi(A_{t_2}^2; Z, W)| & \leq C d_{\infty}^\kappa (A_{t_1}^1,A_{t_2}^2) \\
 | F(A_{t_1}^1, Z, W)  - F(A_{t_2}^2, Z, W) | & \leq C d_{\infty}^\kappa (A_{t_1}^1,A_{t_2}^2).
\end{align*}
Then, from the definition of the viscosity super-solution (ii) in Definition \ref{Definition_5_2}, Lemmas \ref{Lemma_2} and \ref{Lemma_3}, and the property of $\phi$, we have
\begin{align}	
\label{eq_5_12}
& \mathbb{L}(\bar{A}_t^\epsilon; Z_t, W_t) - \phi(\bar{A}_t^\epsilon; Z_t, W_t) \nonumber \\
& = \mathbb{L}(\bar{A}_t; Z_t, W_t) - \phi(\bar{A}_t; Z_t, W_t)  + \mathbb{L}(\bar{A}_t^\epsilon; Z_t, W_t) - \mathbb{L}(\bar{A}_t; Z_t, W_t) \nonumber \\
&~~~ + \phi(\bar{A}_t; Z_t, W_t) - \phi(\bar{A}_t^\epsilon; Z_t, W_t) \nonumber \\
%& \leq C \|\bar{A}_t^\epsilon - \bar{A}_t\|_{\infty} + + \phi(t,\bar{A}_t, Z_t^{1},Z_t^{2}) - \phi(t,\bar{A}_t^\epsilon, Z_t^{1},Z_t^{2}) \\
%%%%%&  \leq C (4 \mu_0 \epsilon \mu^{-1})^{\kappa} + \phi(t,\bar{A}_t, Z_t^{1},Z_t^{2}) - \phi(t,\bar{A}_t^\epsilon, Z_t^{1},Z_t^{2}) 
& \leq C \|\bar{A}_t^\epsilon - \bar{A}_t\|_{\infty} + C \|\bar{A}_t^\epsilon - \bar{A}_t\|_{\infty}^\kappa  \leq C (4 \mu_0 \epsilon \mu^{-1})^{\kappa}.
\end{align}

Note that by (\ref{eq_5_7}), (\ref{eq_5_12}), H\"older inequality, Lemma \ref{Lemma_1} and (\ref{eq_5_4}),
\begin{align}
\label{eq_5_13}
| L^{(1)}  |  & \leq \mathbb{P}(\xi^\epsilon < t+\tau)^{\frac{1}{2}} \mathbb{E} \Bigl [ \bigl | \phi(X_{t+\tau}^{t,\bar{A}_t^\epsilon; Z_t \otimes u^{\epsilon^\prime}, W_t \otimes v}; (Z_t \otimes u^{\epsilon^\prime})_{t+\tau} \nonumber \\
&~~~~~~~~~~ - \mathbb{L}(X_{t+\tau}^{t, \bar{A}_t^\epsilon; Z_{t} \otimes u^{\epsilon^\prime}, W_t \otimes v}; (Z_t \otimes u^{\epsilon^\prime})_{t+\tau}, (W_t \otimes v)_{t+\tau} ) \bigr |^2 \Phi_{t+\tau}^2 \bigl | \mathcal{F}_t \Bigr ]^{\frac{1}{2}}  \nonumber \\
%& \leq \mathbb{P}(\xi^\epsilon < t+\tau)^{\frac{1}{2}} \mathbb{E} \Bigl [ \bigl | \phi(X_{t+\tau}^{t,\bar{A}_t^\epsilon; Z_t \otimes u^{\epsilon^\prime}, W_t \otimes v}; (Z_t \otimes u^{\epsilon^\prime})_{t+\tau},(W_t \otimes v)_{t+\tau}) \nonumber \\
%&~~~~~~~~~~ - \phi(\bar{A}_{t,\tau}^{\epsilon}   ; (Z_t \otimes u^{\epsilon^\prime})_{t+\tau},(W_t \otimes v)_{t+\tau})   \nonumber\\
%&~~~~~~~~~~ + \phi(\bar{A}_{t,\tau}^{\epsilon}   ; (Z_t \otimes u^{\epsilon^\prime})_{t+\tau},(W_t \otimes v)_{t+\tau}) \nonumber \\
%&~~~~~~~~~~ - \mathbb{L}(\bar{A}_{t,\tau }^{\epsilon}; (Z_t \otimes u^{\epsilon^\prime})_{t+\tau}, (W_t \otimes v)_{t+\tau} ) \nonumber \\
%&~~~~~~~~~~ + \mathbb{L}(\bar{A}_{t,\tau }^{\epsilon}; (Z_t \otimes u^{\epsilon^\prime})_{t+\tau}, (W_t \otimes v)_{t+\tau} ) \nonumber \\
%&~~~~~~~~~~ - \mathbb{L}(X_{t+\tau}^{t, \bar{A}_t^\epsilon; Z_{t} \otimes u^{\epsilon^\prime}, W_t \otimes v}; (Z_t \otimes u^{\epsilon^\prime})_{t+\tau}, (W_t \otimes v)_{t+\tau} ) \bigr |^2 \Phi_{t+\tau}^2 \bigl | \mathcal{F}_t \Bigr ]^{\frac{1}{2}}  \nonumber \\
& \leq C(1+\mu_0^6) \Bigl (\tau^{\frac{d}{2}(\frac{1}{2}-\kappa)} \epsilon^{-\frac{d}{2}} + \frac{\tau^{\frac{3}{2}}}{\delta^3} \Bigr )(\tau^{\frac{1}{2}}+ \tau^{\frac{\kappa}{2}} + C (4 \mu_0 \epsilon \mu^{-1})^{\frac{\kappa}{2}}).
\end{align}
From (ii) of Lemma \ref{Lemma_1}, we also have
\begin{align}
\label{eq_5_14}
|L^{(2)}| &\leq 
  \mathbb{E} \Bigl [ \int_t^{t+\tau} \bigl | F(\bar{A}_t^\epsilon, (Z_t \otimes u^{\epsilon^\prime})_r, (W_t \otimes v)_r)  \nonumber \\
&~~~~~~~~~~ - F(\bar{A}_t, (Z_t \otimes u^{\epsilon^\prime})_r, (W_t \otimes v)_r) \bigr | \dd r \bigr | \mathcal{F}_t \Bigr ] \nonumber \\
&~~~ + \mathbb{E} \Bigl [ \int_t^{t+\tau} \bigl |  F(X_r^{t,\bar{A}_t^\epsilon; Z_t \otimes u^{\epsilon^\prime}, W_t \otimes v}, (Z_t \otimes u^{\epsilon^\prime})_r, (W_t \otimes v)_r)\nonumber \\
&~~~~~~~~~~ - F(\bar{A}_t^\epsilon, (Z_t \otimes u^{\epsilon^\prime})_r, (W_t \otimes v)_r) 	 \bigr | \dd r \bigr | \mathcal{F}_t \Bigr ]  \nonumber \\ 
& \leq  C  \tau\mathbb{E} \Bigl [ d_{\infty}^\kappa(X_r^{t,\bar{A}_t^\epsilon; Z_t \otimes u^{\epsilon^\prime}, W_t \otimes v}, \bar{A}_t^\epsilon) | \mathcal{F}_t \Bigr ] +  C \tau d_{\infty}^\kappa(\bar{A}_t^\epsilon, \bar{A}_t)
\nonumber \\
& \leq C \tau^{1 + \frac{\kappa}{2}}  + C\tau (4 \mu_0 \epsilon \mu^{-1})^\kappa,
\end{align}
and
\begin{align}
\label{eq_5_15}
|L^{(3)}| & \leq C \mathbb{E} \Bigl [ \int_t^{t+\tau} (\Phi_r - 1) \dd r \bigl | \mathcal{F}_t \Bigr ]  \leq C \tau \mathbb{E} \Bigl [ \sup_{r \in [t,t+\tau]} |\Phi_r - 1| | \mathcal{F}_t \Bigr ] \leq C \tau^{\frac{3}{2}}.
\end{align}

Hence, by substituting (\ref{eq_5_12})-(\ref{eq_5_15}) into (\ref{eq_5_11_1}), we have
\begin{align*}
&  (4 \mu_0 \epsilon \mu^{-1})^{\kappa} \frac{1}{\tau} + C(1+\mu_0^6) \Bigl (\tau^{\frac{d}{2}(\frac{1}{2}-\kappa)-1} \epsilon^{-\frac{d}{2}} + \frac{\tau^{\frac{3}{2}}}{\delta^3} \Bigr )(\tau^{\frac{1}{2}}+ \tau^{\frac{\kappa}{2}} + C (4 \mu_0 \epsilon \mu^{-1})^{\frac{\kappa}{2}}) \\
&~~~ + C \tau^{\frac{\kappa}{2}} + C(4 \mu_0 \epsilon \mu^{-1})^\kappa + C \tau^{\frac{1}{2}}  + \epsilon^\prime  \\
& \geq \frac{1}{\tau} \mathbb{E} \Bigl [ \int_t^{t+\tau} F(\bar{A}_t, (Z_t \otimes u^{\epsilon^\prime} )_r, (W_t \otimes v)_r) \dd r | \mathcal{F}_t \Bigr ].
\end{align*}
Let $\tau = \mu^{-\frac{\kappa}{2}}$. Then the arbitrariness of $v$ and $\epsilon^\prime$, and the definition of $F$ imply that
\begin{align*}
0 \geq & \limsup_{\mu \rightarrow \infty}  \mathbb{H}^-(\bar{A}_t,Z_t, W_t, (\partial_t \phi^{u,v}, \phi, \partial_x \phi, \partial_{xx} \phi)(\bar{A}_t;Z_t, W_t) ) .
\end{align*}
This shows that (\ref{eq_2_6}) is the viscosity super-solution of (\ref{eq_5_2}).

Next, we prove that (\ref{eq_2_6}) is the viscosity sub-solution of the lower PHJI equation in (\ref{eq_5_2}). From (i) in Definition \ref{Definition_5_2} and Lemma \ref{Lemma_0_5}, for $\phi \in \mathcal{C}^{1,2}_{\kappa}(\Lambda; \hat{\Lambda} )$, $\mu, \mu_0 > 0$, 
\begin{align*}
		0 = (\mathbb{L} - \phi)(\bar{A}_t; Z_t,W_t) = \sup_{A_s \in \mathbb{C}^{\kappa,\mu,\mu_0}} (\mathbb{L} - \phi)(A_s; Z_t,W_t),
\end{align*}
where $\bar{A}_t \in \mathbb{C}^{\kappa,\mu,\mu_0}$. This implies that $\mathbb{L}(\bar{A}_t; Z_t,W_t) = \phi(\bar{A}_t; Z_t,W_t)$, and for $A_t \neq \bar{A}_t$, $\phi (A_t; Z_t,W_t) \geq \mathbb{L} (A_t; Z_t,W_t)$.

From Lemmas \ref{Lemma_2} and \ref{Lemma_3}, $\mathbb{L} \in \mathcal{C}(\Lambda;\hat{\Lambda})$, and due to the definition of the lower value functional, $\mathbb{L}(A_T;Z_T,W_T) \leq  m(A_T)$. Then it is necessary to prove that
\begin{align*}
%\label{eq_5_17}
& \liminf_{\mu \rightarrow \infty}  \mathbb{H}^-(\bar{A}_t,Z_t, W_t, (\partial_t \phi^{u,v}, \phi, \partial_x \phi, \partial_{xx} \phi)(\bar{A}_t;Z_t, W_t) ) \geq 0.
\end{align*}

Now, suppose that this is not true, i.e., there exists a finite $\mu^\prime > 0$ such that for some $\theta > 0$,
\begin{align*}	
\mathbb{H}^-(\bar{A}_t,Z_t, W_t, (\partial_t \phi^{u,v}, \phi, \partial_x \phi, \partial_{xx} \phi)(A_t;Z_t, W_t) ) \leq - \theta < 0,
\end{align*}
where in view of the definition of $F$, $\sup_{v \in \mathrm{V}}	 \inf_{u \in \mathrm{U}} F(\bar{A}_t,Z_t^u,W_t^v)	\leq - \theta < 0$. Note that $\mathrm{V}$ and $\mathrm{U}$ are compact; hence, there exists a measurable function $\gamma : \mathrm{V} \rightarrow \mathrm{U}$ such that for any $v \in \mathrm{V}$ with $|r-t| \leq \tau_0$,
\begin{align}
\label{eq_5_18}
F(\bar{A}_r,Z_r^{\gamma(v)},W_r^v) \leq - \frac{1}{2}\theta.
\end{align}

On the other hand, from (\ref{eq_3_3}) of Theorem \ref{Theorem_1}, we have
\begin{align*}
 & \essinf_{\alpha \in \mathcal{A}[t,t+\tau]} \esssup_{v \in \mathcal{V}[t,t+\tau]} \Pi_{t}^{t,t+\tau,\bar{A}_t^\epsilon;Z_t \otimes \alpha(v), W_{t} \otimes v}  \\
&~~~~~ \Bigl [ \mathbb{L}(X_{t+\tau}^{t,\bar{A}_t^\epsilon; Z_{t} \otimes \alpha(v), W_t \otimes v}; (Z_t \otimes \alpha(v))_{t+\tau}, (W_t \otimes v)_{t+\tau} ) \Bigr ] 	- \mathbb{L}(\bar{A}_t^\epsilon; Z_t, W_t) = 0.
\end{align*}
By defining $\gamma_s(v) := \gamma(v_s(\omega))$ for $(s,\omega) \in [t,T] \times \Omega$, we have $\gamma \in \mathcal{A}[t,t+\tau]$ and $Z_t \otimes \gamma \in \mathcal{A}[0,T]$. This, together with the definition of $\Pi$ and the comparison principle in (iv) of Lemma \ref{Lemma_1}, implies that
\begin{align*}
& \esssup_{v \in \mathcal{V}[t,t+\tau]} \Pi_{t}^{t,t+\tau,\bar{A}_t^\epsilon;Z_t \otimes \gamma(v), W_{t} \otimes v}  \\
& \Bigl [ \phi(X_{t+\tau}^{t,\bar{A}_t^\epsilon; Z_{t} \otimes \gamma(v), W_t \otimes v}; (Z_t \otimes \gamma(v))_{t+\tau}, (W_t \otimes v)_{t+\tau} ) \Bigr ] 	- \phi(\bar{A}_t^\epsilon; Z_t, W_t) \geq 0.
\end{align*}
For each $\epsilon^\prime > 0$, similar to (\ref{eq_4_13_1_1}), we can choose $v^\prime \in \mathcal{V}[t,t+\tau]$ such that
\begin{align*}	
& \Pi_{t}^{t,t+\tau,\bar{A}_t^\epsilon;Z_t \otimes \gamma(v^\prime), W_{t} \otimes v^\prime} \\
&~~~~~  \Bigl [ \phi(X_{t+\tau}^{t,\bar{A}_t^\epsilon; Z_{t} \otimes \gamma(v^\prime), W_t \otimes v^\prime}; (Z_t \otimes \gamma(v^\prime))_{t+\tau}, (W_t \otimes v^\prime)_{t+\tau} ) \Bigr ] - \phi(\bar{A}_t^\epsilon; Z_t, W_t) \geq - \epsilon^\prime \tau.
\end{align*}
Note (\ref{eq_5_7}) and (\ref{eq_5_7_1_1}). Then, similar to (\ref{eq_5_11}), by Lemma \ref{Lemma_0_7}, we have
\begin{align*}
& \frac{1}{\tau} \mathbb{E} \Bigl [  \int_t^{t+\tau} \Phi_{r} F(X_r^{t,\bar{A}_t^\epsilon; Z_t \otimes \gamma(v^\prime), W_t \otimes v^\prime}, (Z_t \otimes \gamma(v^\prime))_r, (W_t \otimes v^\prime)_r) \dd r \bigl | \mathcal{F}_t \Bigr ] \geq - \epsilon^\prime.
\end{align*}
With the same technique as in the super-solution case and the definition of $\gamma$, by letting $\tau \downarrow 0$, the arbitrariness of $\epsilon^\prime$ and (\ref{eq_5_18}) imply that $0 \leq 	  F(\bar{A}_r,Z_r^{\gamma(v)},W_r^v) \leq -\frac{1}{2}\theta$. This induces $\theta \leq 0$, which leads to a contradiction. Hence, (\ref{eq_2_6}) is the viscosity sub-solution of the lower PHJI equation in (\ref{eq_5_2}). 

The proof for the upper value functional $\mathbb{U}$ being a viscosity solution to the upper PHJI equation in (\ref{eq_5_3}) is similar. We complete the proof of the theorem.
\end{proof}

We now discuss the existence of the game value of the SZSDG under the Isaacs condition. Specifically, we introduce the \emph{state and control path-dependent Isaacs condition}: for $(A_t,Z_t,W_t,r,y,p,P) \in  \Lambda \times \hat{\Lambda} \times \mathbb{R} \times \mathbb{R} \times \mathbb{R}^n \times \mathbb{S}^n$,
\begin{align}	
\label{eq_Isaccs_1}
 & \mathbb{H} (A_t,Z_t, W_t, r,y,p,P)  \\
 & := \mathbb{H}^- (A_t,Z_t, W_t, r,y,p,P) = \mathbb{H}^+ (A_t,Z_t, W_t, r,y,p,P). \nonumber
\end{align}
Then the existence of the game value can be stated as follows:
\begin{theorem}\label{Theorem_5_8_8_8}
	Suppose that Assumptions \ref{Assumption_1} and the uniqueness of the viscosity solutions of (\ref{eq_5_2}) and (\ref{eq_5_3}) hold. Under the Isaacs condition in (\ref{eq_Isaccs_1}), the game has a value, i.e., $\mathbb{L}(A_t) = \mathbb{U}(A_t) =: \mathbb{G}(A_t)$, where $\mathbb{G}$ is the unique viscosity solution of the following PHJI equation:
	\begin{align}
	\label{eq_6_4}
	\begin{cases}
	\mathbb{H}(A_t,Z_t, W_t, (\partial_t \mathbb{G}^{u,v}, \mathbb{G}, \partial_x \mathbb{G}, \partial_{xx} \mathbb{G})(A_t;Z_t, W_t) ) \\
	\mathbb{G}(A_T;Z_T, W_T) = m(A_T),~ (A_T,Z_T, W_T) \in \Lambda_T \times \hat{\Lambda}_T.
	\end{cases}
	\end{align}
	\end{theorem}
	\begin{proof}
	In view of Theorem \ref{Theorem_5_4} and the uniqueness assumption, the lower value functional $\mathbb{L}$ and the upper value functional $\mathbb{U}$ are the unique viscosity solutions of (\ref{eq_5_2}) and (\ref{eq_5_3}), respectively. Then, the Isaacs condition in (\ref{eq_Isaccs_1}) implies $\mathbb{L}(A_t) = \mathbb{U}(A_t) =: \mathbb{G}(A_t)$, which is the unique solution to the PHJI equation in (\ref{eq_6_4}). We complete the proof.
\end{proof}

Before concluding the paper, we discuss the state path-dependent case, which is a special case of the SZSDG in this paper and was studied in \cite{Zhang_SICON_2017}.\footnote{As mentioned in Section \ref{Section_1}, \cite{Zhang_SICON_2017} considered the existence of the game value and the approximated saddle-point equilibrium, both under the Isaacs condition, but did not study the existence and uniqueness of (viscosity or classical) solutions of state path-dependent HJI equations.} Specifically, as stated in Remarks \ref{Remark_3_7} and \ref{Remark_5_1_1}, we need to assume that
\begin{assumption}\label{Assumption_2}
	$f:\Lambda \times \mathrm{U} \times \mathrm{V} \rightarrow \mathbb{R}^n$, $\sigma :\Lambda \times \mathrm{U} \times \mathrm{V} \rightarrow \mathbb{R}^{n \times p}$, $l :\Lambda \times \mathbb{R} \times \mathbb{R}^{1 \times p} \times \mathrm{U} \times \mathrm{V} \rightarrow \mathbb{R} $. 
\end{assumption}

\begin{remark}
With Assumption \ref{Assumption_2}, (\ref{eq_Isaccs_1}) becomes the state path-dependent Isaacs condition in \cite[(3.2)]{Zhang_SICON_2017}. Hence, Theorem \ref{Theorem_5_8_8_8} is reduced to \cite[Theorem 3.1]{Zhang_SICON_2017} when Assumption \ref{Assumption_2} holds.
\end{remark}

Under Assumption \ref{Assumption_2}, the lower and upper PHJI equations in (\ref{eq_5_2}) and (\ref{eq_5_3}) are reduced to the state path-dependent HJI equations (see Remark \ref{Remark_5_1_1}):
	\begin{align}
	\label{eq_6_1}
	& \begin{cases}
		\partial_t \mathbb{L}(A_t) + \sup_{v \in \mathrm{V}}	 \inf_{u \in \mathrm{U}}  \mathcal{H}(A_t,u,v,(\mathbb{L}, \partial_x \mathbb{L}, \partial_{xx} \mathbb{L})(A_t) )  = 0 ,\\
	~~~~~~~~~~ t \in [0,T),~ A_t \in \Lambda \\
	\mathbb{L}(A_T) = m(A_T),~ A_T \in \Lambda_T,
	\end{cases}
	\end{align}
	and
	\begin{align} 	
	\label{eq_6_2}
	& \begin{cases}
		\partial_t \mathbb{U}(A_t) +  \inf_{u \in \mathrm{U}} \sup_{v \in \mathrm{V}}	  \mathcal{H}(A_t,u,v,(\mathbb{U}, \partial_x \mathbb{U}, \partial_{xx} \mathbb{U})(A_t) )  = 0 ,\\
	~~~~~~~~~~ t \in [0,T),~ A_t \in \Lambda \\
	\mathbb{U}(A_T) = m(A_T),~ A_T \in \Lambda_T.
	\end{cases}	
	\end{align}

	\begin{assumption}\label{Assumption_3}
	Let $\bar{\mathcal{H}}(A_t,y,p,P) := \sup_{v \in \mathrm{V}}	\inf_{u \in \mathrm{U}}  \mathcal{H}(A_t,u,v,y,p,P)$ and \\$\tilde{\mathcal{H}}(A_t,y,p,P) :=   \inf_{u \in \mathrm{U}} \sup_{v \in \mathrm{V}}	   \mathcal{H}(A_t,u,v,y,p,P)$. For any $(A_t,p) \in \Lambda \times \mathbb{R}$, $y_1,y_2 \in \mathbb{R}$ and $P_1,P_2 \in \mathbb{S}^n$ with $y_1 \geq y_2$ and $P_1 \leq P_2$, 
		\begin{align*}
			\bar{\mathcal{H}}(A_t,y_1,p,P_1) \leq \bar{\mathcal{H}}(A_t,y_2,p,P_2),~ \tilde{\mathcal{H}}(A_t,y_1,p,P_1) \leq \tilde{\mathcal{H}}(A_t,y_2,p,P_2).
		\end{align*}
	\end{assumption}
	
We state the uniqueness of classical solutions of (\ref{eq_6_1}) and (\ref{eq_6_2}).
\begin{proposition}\label{Theorem_5_10}
		Assume that Assumptions \ref{Assumption_1}, \ref{Assumption_2} and \ref{Assumption_3} hold. Suppose that $\mathbb{L}_1$ and $\mathbb{L}_2$ are classical sub- and super-solutions of the lower PHJI equation in (\ref{eq_6_1}), respectively. Then we have $\mathbb{L}_1(A_t) \leq \mathbb{L}_2(A_t)$ for $A_t \in \Lambda$. The same result holds for the upper PHJI equation in (\ref{eq_6_2}). Consequently, there is a unique classical solution of (\ref{eq_6_1}) and (\ref{eq_6_2}).
	\end{proposition}
\begin{proof}
Let $\bar{\mathbb{L}}_1(A_t) := \mathbb{L}_1(A_t) - \frac{\delta}{t}$, where $\delta > 0$. Then we can easily see that $\bar{\mathbb{L}}$ is a classical sub-solution of the following PDE (see (2) of Remark \ref{Remark_5_1}):
\begin{align*}
	\begin{cases}
		\partial_t \bar{\mathbb{L}}_1(A_t) +   \bar{\mathcal{H}}(A_t,(\bar{\mathbb{L}}_1, \partial_x \bar{\mathbb{L}}_1, \partial_{xx} \bar{\mathbb{L}}_1)(A_t) )  \geq  \frac{\delta}{t^2} ,~ t \in [0,T),~ A_t \in \Lambda \\
	\bar{\mathbb{L}}_1(A_T) = m(A_T) - \frac{\delta}{T},~ A_T \in \Lambda_T.
	\end{cases}	
\end{align*}
Since $\mathbb{L}_1 \leq \mathbb{L}_2$ follows from $\bar{\mathbb{L}}_1 \leq \mathbb{L}_2$ in the limit $\delta \downarrow 0$, it suffices to prove the theorem with the following additional assumption:
\begin{align*}	
		\partial_t \mathbb{L}_1(A_t) +   \bar{\mathcal{H}}(A_t,(\mathbb{L}_1, \partial_x \mathbb{L}_1, \partial_{xx} \mathbb{L}_1)(A_t) )  \geq  \nu > 0,
\end{align*}
where $\nu := \frac{\delta}{t^2}$ and $\lim_{t \rightarrow 0} \mathbb{L}_1(A_t) = -\infty$ uniformly on $[0,T)$.

Assume that this is not true, that is, there exists $A_{t^\prime}^\prime \in \Lambda$ with $t^\prime \in [0,T]$ such that $k^\prime := \mathbb{L}_1(A_{t^\prime}^\prime) - \mathbb{L}_2(A_{t^\prime}^\prime) > 0$. In view of Lemma \ref{Lemma_0_5}, there exists $\bar{A}_{\bar{t}} \in \mathbb{C}^{\kappa,\mu,\mu_0}$ with $\bar{t} \in [0,T]$ such that $\mathbb{L}_1 (\bar{A}_{\bar{t}}) - \mathbb{L}_2(\bar{A}_{\bar{t}}) = \sup_{A_t \in \mathbb{C}^{\kappa,\mu,\mu_0}} \mathbb{L}_1 (A_t) - \mathbb{L}_2(A_t) \geq k^\prime$. 
%\begin{align*}
%\mathbb{L}_1 (\bar{A}_{\bar{t}}) - \mathbb{L}_2(\bar{A}_{\bar{t}}) = \sup_{A_t \in \mathbb{C}^{\kappa,\mu,\mu_0}} \mathbb{L}_1 (A_t) - \mathbb{L}_2(A_t) \geq k^\prime.	
%\end{align*}
Then from \cite[Lemma 9]{Peng_Gexpectation_arXiv_2011}, we have $\partial_t 	(\mathbb{L}_1 - \mathbb{L}_2)(\bar{A}_{\bar{t}}) \leq 0, ~\partial_x 	(\mathbb{L}_1 - \mathbb{L}_2)(\bar{A}_{\bar{t}}) = 0$ and $\partial_{xx} 	(\mathbb{L}_1 - \mathbb{L}_2)(\bar{A}_{\bar{t}}) \leq  0$.
%\begin{align*}
%\partial_t 	(\mathbb{L}_1 - \mathbb{L}_2)(\bar{A}_{\bar{t}}) \leq 0, ~\partial_x 	(\mathbb{L}_1 - \mathbb{L}_2)(\bar{A}_{\bar{t}}) = 0,~ \partial_{xx} 	(\mathbb{L}_1 - \mathbb{L}_2)(\bar{A}_{\bar{t}}) \leq  0.
%\end{align*}
This, together with Assumption \ref{Assumption_3} and the fact that $\mathbb{L}_2$ is the classical super-solution, implies that
\begin{align*}	
0  &\geq  \partial_t \mathbb{L}_2(\bar{A}_{\bar{t}}) + \bar{\mathcal{H}}(\bar{A}_{\bar{t}}, (\mathbb{L}_2,\partial_x \mathbb{L}_2,\partial_{xx} \mathbb{L}_2)(\bar{A}_{\bar{t}})) \\
& \geq \partial_t \mathbb{L}_1(\bar{A}_{\bar{t}}) + \bar{\mathcal{H}}(\bar{A}_{\bar{t}}, (\mathbb{L}_1,\partial_x \mathbb{L}_1,\partial_{xx} \mathbb{L}_1)(\bar{A}_{\bar{t}})) \geq  \nu > 0,
\end{align*}
which induces a contradiction. Hence, $\mathbb{L}_1(A_t) \leq \mathbb{L}_2(A_t)$ for $A_t \in \Lambda$. Suppose that $\hat{\mathbb{L}}$ and $\tilde{\mathbb{L}}$ are classical solutions of (\ref{eq_6_1}). Then we have $\hat{\mathbb{L}} \leq \tilde{\mathbb{L}}$ and $\hat{\mathbb{L}} \geq \tilde{\mathbb{L}}$, which implies $\mathbb{L} := \hat{\mathbb{L}} = \tilde{\mathbb{L}}$. Hence, the uniqueness follows. This completes the proof.
%The same argument can be applied to the upper PHJI equation. This completes the proof.
\end{proof}

%
%
%Based on Theorems \ref{Theorem_5_8} and \ref{Theorem_5_10}, we have the following corollary:
%\begin{corollary}
%Suppose that Assumptions \ref{Assumption_1}, \ref{Assumption_2}, \ref{Assumption_3} and the Isaacs condition in (\ref{eq_6_3}) hold, and that $\mathbb{L}$ and $\mathbb{U}$ are classical solutions of (\ref{eq_6_1}) and (\ref{eq_6_2}), respectively. Then the game has a value, i.e., $\mathbb{L}(A_t) = \mathbb{U}(A_t) := \mathbb{G}(A_t)$, and $\mathbb{G}$ is the unique classical solution of (\ref{eq_6_4}).
%\end{corollary}

There are several interesting potential future research problems. 

One important problem is the uniqueness of viscosity solutions of the (lower and upper) PHJI equations in (\ref{eq_5_2}) and (\ref{eq_5_3}). As mentioned in Section \ref{Section_1}, the uniqueness has not been shown even in the case of (strong and weak formulation) state path-dependent SZSDGs \cite{Possamai_arXiv_2018, Pham_SICON_2014, Zhang_SICON_2017}.

%
%(see \cite[Section 6]{Pham_SICON_2014} and \cite[Remark 3.7]{Possamai_arXiv_2018}). Note that \cite{Possamai_arXiv_2018} did not prove the uniqueness of viscosity solutions. As mentioned in  \cite[p. 10]{Possamai_arXiv_2018}, the major motivation of SZSDGs in weak formulation is to study the existence of the saddle-point equilibrium; however, it requires more stringent assumptions on the coefficients than SZSDGs in strong formulation. Zhang in \cite{Zhang_SICON_2017}

%Under the different definitions of viscosity solutions, the uniqueness of path-dependent PDEs was shown in \cite{Ekren_AP_2016, Peng_Gexpectation_arXiv_2011, Pham_SICON_2014}. We expect to use the path-frozen PDE approach as in \cite{Ekren_AP_2016,Pham_SICON_2014}  or the left-frozen approach in \cite{Peng_Gexpectation_arXiv_2011}, where the latter approach might be useful, since our set $\mathbb{C}^{\kappa,\mu,\mu_0}$ is compact due to Lemma \ref{Lemma_0_5}. We may need additional technical assumptions to prove the uniqueness as in \cite{Pham_SICON_2014} (see also \cite[Remark 3.7]{Possamai_arXiv_2018}). 

Another problem is the existence of the (approximated) saddle-point equilibrium using the notion of nonanticipative strategies with delay as mentioned in (2) of Remark \ref{Remark_3_6}. For the state dependent case (with Assumption \ref{Assumption_2}), this was shown in \cite[Theorem 4.13]{Zhang_SICON_2017} under the Isaacs condition, where the key step is approximating the PHJI equation in (\ref{eq_6_1}) and (\ref{eq_6_2}) to the state-dependent (not state path-dependent) HJI equations. Note that there is a unique viscosity solution of the approximated (lower and upper) state-dependent HJI equations in view of \cite[Theorem 5.3]{Buckdahn_SICON_2008}. Then, the existence of the (approximated) saddle-point equilibrium can be shown using the property of nonanticipative strategies with delay \cite[Lemma 2.4]{Buckdahn_SICON_2004}. We speculate that the approach of \cite[Theorem 4.13]{Zhang_SICON_2017} can be applied to the state and control path-dependent SZSDG of this paper. 

We can consider the problem in weak formulation. As noted in (3) of Remark \ref{Remark_3_6}, one major feature of this formulation is the symmetric feedback information between the players, which is convenient to show the existence of the saddle-point equilibrium and the game value. 

%However, in order to show the existence and uniqueness of viscosity solutions, additional technical assumptions are required \cite{Pham_SICON_2014, Possamai_arXiv_2018}.

Finally, the forward-backward stochastic differential equation given in (\ref{eq_2_1}) and (\ref{eq_2_2}) is not fully coupled in the sense that the BSDE in (\ref{eq_2_2}) is not included in the (forward) SDE in (\ref{eq_2_1}). This can be extended to the fully-coupled FBSDE, where (\ref{eq_2_1}) is also dependent on (\ref{eq_2_2}). This can be viewed as a generalization of \cite{Li_SICON_2014}, where the major challenge is the case when the diffusion term of (\ref{eq_2_1}) depends on the second component of the solution of the BSDE, since the associated PHJI equation should require an additional algebraic equation.

\bibliographystyle{siam}
\bibliography{researches_1}
\end{document}